\documentclass{smfart}
\usepackage{amsmath, amssymb, amsthm, amscd,smfenum,xspace,mathrsfs,url,upref,verbatim}
\usepackage[utf8]{inputenc} 
\usepackage[T1]{fontenc}
\usepackage{dsfont}
\usepackage{rotating}
\usepackage{tikz-cd}
\usepackage{tikz}
\usetikzlibrary{calc}
\usepackage{pdflscape}
\usepackage{stmaryrd}
\usepackage{lscape}
\usepackage[toc,page]{appendix} 

\usepackage[all,cmtip]{xy}
\usepackage{xfrac} 

\usepackage{ textcomp }

\let\hat=\widehat
\let\tilde=\widetilde
\numberwithin{equation}{subsection}

\newtheorem{soustheorem}[equation]{Théorème} 
\newtheorem{theorem}[]{Théorème} 

\newtheorem{proposition}[equation]{Proposition}
\newtheorem{lemme}[equation]{Lemme}
\newtheorem{corollaire}[equation]{Corollaire}

\makeatother

\let\leq\leqslant
\let\geq\geqslant

\theoremstyle{remark}

\DeclareMathOperator{\CK}{CK}
\DeclareMathOperator{\codim}{codim}

\DeclareMathOperator{\DR}{DR}
\DeclareMathOperator{\irr}{irr}

\DeclareMathOperator{\reg}{reg}

\DeclareMathOperator{\Sh}{Sh}
\DeclareMathOperator{\Supp}{Supp}
\DeclareMathOperator{\Diag}{Diag}
\DeclareMathOperator{\rg}{rg}
\DeclareMathOperator{\car}{car}

\def\cartesien{\ar@{}[rd]|{\square}}

\DeclareMathOperator{\Irr}{Irr}

\DeclareMathOperator{\hol}{hol}
\DeclareMathOperator{\sing}{sing}

\title{Sur une caractérisation des $\mathcal{D}$-modules holonomes réguliers}

\author[J.-B.~Teyssier]{Jean-Baptiste Teyssier\thanks{Ce travail a été financé par la fondation Einstein, ainsi que par la bourse ERC 226257.}}

\date{}
\curraddr{Freie Universität Berlin, Mathematisches Institut, Arnimallee 3, 14195 Berlin, Germany}
\email{teyssier@zedat.fu-berlin.de}

\usepackage[french]{babel}
\begin{document}

\maketitle


\section*{Introduction}
Soit $X$ une variété complexe lisse et $\mathcal{M}_1$,$\mathcal{M}_2$ deux objets de la catégorie dérivée $D^{b}_{\hol}(\mathcal{D}_X)$ formée des complexes de $\mathcal{D}_X$-modules à cohomologie bornée et holonome. On note $\Sh_{\mathds{C}}(X)$ la catégorie des faisceaux en $\mathds{C}$-espaces vectoriels sur $X$ et $D^{b}_{c}(X,\mathds{C})$ la catégorie dérivée formée des complexes d'objets de $\Sh_{\mathds{C}}(X)$ à cohomologie bornée et constructible. \\ \indent
Traditionnellement \cite[6.1.1]{KashKawai}\cite[2.1.1]{Mebuneautre}, la pleine fidélité de la correspondance de \- Riemann-Hilbert se démontre en prouvant que le morphisme 
\begin{equation}\label{RH}
\xymatrix{
RH_{\mathcal{M}_1,\mathcal{M}_2}: R\mathcal{H}om_{\mathcal{D}_X}(\mathcal{M}_1,\mathcal{M}_2)\ar[r]     & R\mathcal{H}om_{\Sh_{\mathds{C}}(X)}(\mathbf{S}(\mathcal{M}_2),\mathbf{S}(\mathcal{M}_1))
}
\end{equation}
est un isomorphisme de $D^{b}_{c}(X,\mathds{C})$ dès que les complexes $\mathcal{M}_1$ et $\mathcal{M}_2$ sont à cohomologie régulière. Le but de cet article est de prouver le 
\begin{theorem}\label{theo1}
Soit $\mathcal{M}\in D^{b}_{\hol}(\mathcal{D}_X)$. Si le morphisme de Riemann-Hilbert $RH_{\mathcal{M},\mathcal{M}}$ est un isomorphisme, alors $\mathcal{M}$ est à cohomologie régulière.
\end{theorem}
\noindent
Il s'agit d'une caractérisation de la régularité qui, bien qu'invé\-rifiable en pratique, a  certaines conséquences méta-mathématiques qui mériteraient d'être approfondies, comme par exemple le fait que si l'on cherche à définir une notion d'objet régulier dans quelque contexte que ce soit, il faut d'abord chercher à définir un analogue au foncteur solution. Pour un exemple formellement proche des $\mathcal{D}$-modules, on peut mentionner  \cite{EK}. \\ \indent
Détaillons la stratégie de la preuve du théorème \ref{theo1}. Notons $\Delta_X$ la diagonale de $X\times X$ et désignons par $\Irr^{\ast}_{\Delta_X}: D^{b}_{\hol}(\mathcal{D}_{X}) \longrightarrow D^{b}_{c}(X,\mathds{C})$  le foncteur d'irrégularité \cite{Mehbgro} le long de $\Delta_X$. Notons enfin $\mathcal{M}^{\vee}$ le complexe dual de $\mathcal{M}$ et pour $\mathcal{M}_1$,$\mathcal{M}_2\in D^{b}_{\hol}(\mathcal{D}_X)$, notons $\mathcal{M}_{1}\boxtimes \mathcal{M}_{2}$ le produit tensoriel externe de $\mathcal{M}_{1}$ et $\mathcal{M}_{2}$. \\ \indent L'analyse\footnote{Entièrement due à Mebkhout, et rappelée ici en partie \ref{Rappel}.} du cône de $RH_{\mathcal{M}_1,\mathcal{M}_2}$ permet de voir que le théorème \ref{theo1} est équivalent au 
\begin{theorem}\label{theo2prime}
Si $\Irr^{\ast}_{\Delta_X}(\mathcal{M} \boxtimes \mathcal{M}^{\vee})$ est l'objet nul de $D^{b}_{c}(X,\mathds{C})$, alors $\mathcal{M}$ est à cohomologie régulière.
\end{theorem}
Pour prouver ce théorème, on raisonne par récurrence sur la dimension de $X$. le cas où $X$ est une courbe est traité en \ref{cascourbe} via une analyse détaillée des triangles \eqref{irrdiagetponct} et \eqref{triangle2} calculant l'irrégularité diagonale.\\ \indent
En dimension $>1$, on commence en \ref{regensdisct} par se ramener au cas où $\mathcal{M}$ est régulier en dehors d'un point $x_0$.  Cela se fait via un résultat \ref{restrcar} de commutation de $\Irr^{\ast}$ avec la restriction non-caractéristique.\\ \indent
A l'aide d'une technique de récurrence descendante sur la dimension du support due à Mebkhout, on prouve en \ref{reduction}  l'existence d'une courbe $C$ (éventuellement singulière en $x_0$) qui est telle que la régularité de $\mathcal{M}$ équivaut à la régularité de l'un des complexes $\mathcal{M}_1:=R\Gamma_{[C]}\mathcal{M}^{\vee}$ et $\mathcal{M}_2:=R\Gamma_{[C]}\mathcal{M}$. La clé de voûte de cette technique est la généralisation au cas des modules holonomes du critère de Deligne \cite{Del} stipulant que pour tester la régularité d'une connexion méromorphe, il suffit de le faire génériquement le long du lieu des pôles. C'est le \textit{critère fondamental de la régularité} \cite[4.3-16]{Mehbsmf}.
Dans le contexte des modules holonomes, ce résultat n'est efficient que dans le cas où le module étudié est régulier sur un ouvert assez gros et à support de dimension $>1$.\\ \indent
Pour démontrer que $\mathcal{M}_1$ ou $\mathcal{M}_2$ est régulier, on fait la distinction entre les composantes de $C$ le long desquelles $\mathcal{M}_1$ admet un module de cohomologie irrégulier  et les composantes le long desquelles tous les $\mathcal{H}^{i}\mathcal{M}_1$ sont réguliers. Ces dernières composantes sont non pertinentes pour l'étude de la régularité de $\mathcal{M}_1$. On montre en \ref{elimination} qu'on peut toujours les éliminer du problème, et de même avec $\mathcal{M}_2$. On raisonne alors par l'absurde: si les composantes de $C$ ne peuvent pas être toutes éliminées, c'est que $C$ admet une composante le long de laquelle deux modules de cohomologie $\mathcal{H}^{i}\mathcal{M}_1$ et  $\mathcal{H}^{j}\mathcal{M}_2$ sont irréguliers.\\ \indent
Par argument  de normalisation, on aboutit en \ref{reddim1} à une situation on l'on dispose de deux modules holonomes irréguliers $\mathcal{N}_1$ et $\mathcal{N}_2$ définis sur un voisinage de $0$ dans $\mathds{C}$ et tels que $\Irr^{\ast}_{\Delta_{\mathds{C}}}(\mathcal{N}_1 \boxtimes \mathcal{N}_2)\simeq 0$. Cette situation est contradictoire avec l'analyse menée en \ref{cascourbe}, ce qui conclut la preuve du théorème \ref{theo2prime}. Un ingrédient essentiel du passage de $C$ à sa normalisation est une estimation \ref{lemme0} de l'amplitude de l'irrégularité diagonale pour un produit tensoriel externe de modules à support une courbe. Cette estimation repose de façon cruciale sur le théorème de perversité du faisceau d'irrégularité le long d'une hypersurface \cite[2.1.6]{Mehbgro}.  
\\ \indent

Je remercie Claude Sabbah pour de multiples discussions concernant ce travail. Je remercie aussi Hélène Esnault pour ses encouragements durant l'élaboration de cet article, ainsi qu'Yves Laurent et Stéphane Guillermou pour m'avoir signalé une erreur dans une version antérieure de ce texte.

\section{Notations et rappels}
Les résultats apparaissant dans cette section seront utilisés à plusieurs reprises dans la suite de ce texte. On les récapitule ici une fois pour toute. La lettre $X$ désigne une variété complexe lisse et sauf mention du contraire,  $\mathcal{M}$ désigne un objet de $D^{b}_{\hol}(\mathcal{D}_X)$.  On dira que  $\mathcal{M}$ est \textit{régulier} si ses modules de cohomologie sont réguliers. Enfin, on notera $\car \mathcal{M}$ la \textit{variété caractéristique} de $\mathcal{M}$.
\subsection{}\label{commutationsolbox}
Notons $\DR:D^{b}_{\hol}(\mathcal{D}_X) \longrightarrow D^{b}_{c}(X,\mathds{C})$ le \textit{foncteur de De Rham} et désignons par $\mathbf{S}:D^{b}_{\hol}(\mathcal{D}_X) \longrightarrow D^{b}_{c}(X,\mathds{C})$ le \textit{foncteur solution}. On rappelle que si 
$\mathcal{M}_1,\mathcal{M}_2\in  D^{b}_{\hol}(\mathcal{D}_X)$, on a suivant \cite[1.4.3]{KashKawai}  une identification canonique
$$
\mathbf{S}(\mathcal{M}_{1}\boxtimes \mathcal{M}_{2})\overset{\sim}{\longrightarrow}
\mathbf{S}(\mathcal{M}_{1})\boxtimes_{\mathds{C}} \mathbf{S}(\mathcal{M}_{2})
$$
\subsection{}Pour tout morphisme $f:Y\longrightarrow X$ avec $Y$ variété complexe lisse, on note $f^{+}:D^{b}_{\hol}(\mathcal{D}_X)\longrightarrow D^{b}_{\hol}(\mathcal{D}_Y)$ et $f_{+}:D^{b}_{\hol}(\mathcal{D}_Y)\longrightarrow D^{b}_{\hol}(\mathcal{D}_X)$ les foncteurs image inverse et image directe pour les $\mathcal{D}$-modules. On notera $f^{\dag}$ pour $f^{+}[\dim Y-\dim X]$.
\subsection{}\label{localisation}
Pour tout sous-espace analytique fermé $Z$ de $X$, on note $i_Z:Z\hookrightarrow X$ l'inclusion canonique. Suivant \cite[IV]{MT}, on dispose du \textit{triangle de cohomologie locale algébrique}
\begin{equation}\label{ocholocal}
\xymatrix{
R\Gamma_{[Z]}\mathcal{M}\ar[r]&  \mathcal{M} \ar[r]& R\mathcal{M}(\ast Z)\ar[r]^-{+1}& 
}
\end{equation}
Il s'agit d'un triangle distingué de $D^{b}_{\hol}(\mathcal{D}_X)$.
Dans ce triangle, le complexe $R\Gamma_{[Z]}\mathcal{M}$ est \textit{la cohomologie locale algébrique} de $\mathcal{M}$ le long de $Z$ et $R\mathcal{M}(\ast Z)$ est la \textit{localisation} de $\mathcal{M}$ le long de $Z$. \\ \indent
Pour tout morphisme $f:Y\longrightarrow X$, on a un isomorphisme canonique de triangles distingués  \cite[IV.4.1]{MT}
\begin{equation}\label{cohlociminv}
\xymatrix{
 f^{+}R\Gamma_{[Z]}\mathcal{M}\ar[r] \ar[d]^-\wr&   f^{+}\mathcal{M}   \ar[r] \ar[d]^-\wr&   f^{+}R\mathcal{M}(\ast Z)  \ar[d]^-\wr \\
R\Gamma_{[f^{-1}(Z)]}f^{+}\mathcal{M}    \ar[r] &      f^{+}\mathcal{M}        \ar[r] &   R(f^{+}\mathcal{M})(\ast f^{-1}(Z))
}
\end{equation}
Si de plus $Z$ est lisse, on a une identification canonique \cite[IV.5.1]{MT} 
\begin{equation}\label{calculcohlocale}
R\Gamma_{[Z]}\mathcal{M} \overset{\sim}{\longrightarrow}  i_{Z+}i^{\dag}_{Z}\mathcal{M}.
\end{equation}
\begin{lemme}\label{coholocalsupport}
Soit $\mathcal{M}\in D^{b}_{\hol}(\mathcal{D}_X)$  à support contenu dans un sous-espace analytique fermé $Z$ de $X$. Alors on a $R\mathcal{M}(\ast Z)\simeq 0$.
\end{lemme}
\begin{proof}
Sur le lieu de lissité de $Z$, le lemme \ref{coholocalsupport} est une conséquence immédiate du théorème de Kashiwara \cite[1.6.1]{HTT} et de la formule \eqref{calculcohlocale}. Autrement dit, le complexe $R\mathcal{M}(\ast Z)$ est à support contenu dans le lieu singulier $Z^{\sing}$ de $Z$. Du fait de l'identité
$$
R\mathcal{M}(\ast Z)\simeq R(R\mathcal{M}(\ast Z)(\ast Z^{\sing}))
$$
on est ramené à démontrer le lemme \ref{coholocalsupport} pour $R\mathcal{M}(\ast Z)$ et le sous-espace $Z^{\sing}$. En raisonnant par récurrence descendante sur la dimension de $Z$, on voit qu'il suffit de traiter le cas où $Z$ est un point, auquel cas on applique de nouveau le théorème de Kashiwara et la formule \eqref{calculcohlocale}.
\end{proof}
On prendra garde que la cohomologie locale algébrique et la localisation ne commutent pas à la dualité. Néanmoins, on a le
\begin{corollaire}\label{pticoro}
Soit $\mathcal{M}\in D^{b}_{\hol}(\mathcal{D}_X)$ et soit $Z$ un sous-espace analytique fermé de $X$. On a une identification canonique
$$
R\mathcal{M}^{\vee}(\ast Z)\simeq R(R\mathcal{M}(\ast Z))^{\vee}(\ast Z)
$$
\end{corollaire}
\begin{proof}
En appliquant le foncteur de dualité au triangle de cohomologie locale \eqref{ocholocal}, on obtient un triangle 
\begin{equation}\label{coholcaldual}
\xymatrix{
R\mathcal{M}(\ast Z)^{\vee} \ar[r]&  \mathcal{M}^{\vee} \ar[r]& (R\Gamma_{[Z]}\mathcal{M})^{\vee}\ar[r]^-{+1}& 
}
\end{equation}
Puisque la dualité préserve le support, le complexe $ (R\Gamma_{[Z]}\mathcal{M})^{\vee}$ est à support contenu dans $Z$. Par application du foncteur de localisation à \eqref{coholcaldual}, on obtient \ref{pticoro} immédiatement à partir de \ref{coholocalsupport}.
\end{proof}
\subsection{}\label{compproduittens}
Les foncteurs de cohomologie locale algébrique et localisation sont compatibles au produit tensoriel externe.
Soient $X_1$ et $X_2$ deux variétés complexes. Pour $i=1,2$, soit $\mathcal{M}_i\in  D^{b}_{\hol}(\mathcal{D}_{X_i})$ et soit  $Y_i$ un sous-espace analytique de $X_i$. On pose $Z_1=Y_1\times X_2$ et $Z_2=X_1\times Y_2$. Alors d'après \cite[IV.1]{MT}, on a  des identifications canoniques
\begin{align*}
R(\mathcal{M}_1\boxtimes \mathcal{M}_2)(\ast Z_1 \cup Z_2)&\simeq R\mathcal{M}_1(\ast Y_1)\boxtimes R\mathcal{M}_2(\ast Y_2)\\
R\Gamma_{[Y_1\times Y_2]}(\mathcal{M}_1\boxtimes \mathcal{M}_2)&\simeq R\Gamma_{[Y_1]}\mathcal{M}_1\boxtimes R\Gamma_{[Y_2]}\mathcal{M}_2
\end{align*}

\subsection{}\label{defirr}
Si $Z$ est un sous-espace analytique fermé de $X$, on pose $$\Irr^{\ast}_{Z}(\mathcal{M}):=i_{Z\ast}i_Z^{-1}\mathbf{S}(R\mathcal{M}(\ast Z))$$ 
C'est le \textit{faisceau d'irrégularité de $\mathcal{M}$ le long de $Z$}. En appliquant $i_{Z\ast}i_Z^{-1}\mathbf{S}$ à \eqref{ocholocal}, on voit qu'il s'insère dans le triangle distingué de $D^{b}_{c}(X,\mathds{C})$

\begin{equation}\label{defirrprime}
\xymatrix{
\Irr^{\ast}_{Z}(\mathcal{M})\ar[r]&  i_{Z\ast}i_Z^{-1}\mathbf{S}(\mathcal{M})\ar[r]& \mathbf{S}(R\Gamma_{[Z]}\mathcal{M}) \ar[r]^-{+1}& 
}
\end{equation}
Notons $\mathcal{O}_{X\hat{|}Z}$ la formalisation de $\mathcal{O}_{X}$ le long de $Z$, et posons $\mathcal{O}_{X|Z}:=i_Z^{-1}\mathcal{O}_{X}$ et $\mathcal{Q}_Z:=\mathcal{O}_{X\hat{|}Z}/\mathcal{O}_{X|Z}$, alors d'après \cite[3.4-4]{Mehbsmf}, le triangle \eqref{defirrprime} est isomorphe au triangle
$$
\xymatrix{
R\mathcal{H}om_{\mathcal{D}_X}(\mathcal{M}, \mathcal{Q}_Z  )[-1]\longrightarrow R\mathcal{H}om_{\mathcal{D}_X}(\mathcal{M}, \mathcal{O}_{X|Z} )\longrightarrow R\mathcal{H}om_{\mathcal{D}_X}(\mathcal{M}, \mathcal{O}_{X\hat{|}Z} ) \overset{+1}{\longrightarrow}& 
}
$$
Dans la suite, on notera $\mathbf{S}(\mathcal{M},\hat{\mathcal{O}})$  pour $R\mathcal{H}om_{\mathcal{D}_X}(\mathcal{M}, \mathcal{O}_{X\hat{|}Z})$ quand aucune confusion n'est à craindre quant au sous-espace $Z$ choisi.
\subsection{}\label{compmorphismepropre}
La formation du faisceau d'irrégularité est compatible aux morphismes propres. Soit en effet $f:Y\longrightarrow X$ un morphisme propre entre variétés complexes lisses, et soit $Z$ un sous-espace analytique de $X$. Alors on a une identification canonique \cite[3.6-6]{Mehbsmf}
$$
\Irr^{\ast}_{Z}(f_+\mathcal{M})[\dim X]\overset{\sim}{\longrightarrow}Rf_{\ast}\Irr^{\ast}_{f^{-1}(Z)}(\mathcal{M})[\dim Y]
$$  

\subsection{}\label{CK}
Si $f:Y\longrightarrow X$ est un morphisme avec $Y$ variété lisse, on dispose d'un morphisme de comparaison canonique
$$
\CK_{f,\mathcal{M}}:f^{-1}\mathbf{S}(\mathcal{M})\longrightarrow \mathbf{S}(f^{+}\mathcal{M}) 
$$
appelé \textit{morphisme de Cauchy-Kowalevski}. D'après \cite{TheseKashiwara}, il s'agit d'un isomorphisme si $f$ est \textit{non-caractéristique} pour $\mathcal{M}$. \\ \indent
Dans le cas d'une immersion fermée $i_{Y}:Y\hookrightarrow X$, l'analyse détaillée de  $\CK_{i_{Y},\mathcal{M}}$ faite dans \cite[V 2.2]{MT} montre que $\CK_{i_{Y},\mathcal{M}}$ est un isomorphisme si et seulement si $\Irr^{\ast}_{Y}(\mathcal{M})$ est nul.
\subsection{}\label{lemmemeb}
On dira qu'un module holonome $\mathcal{M}$ est \textit{lisse} si le support $\Supp(\mathcal{M})$ de $\mathcal{M}$ est lisse équidimensionnel et si la variété caractéristique $\car \mathcal{M}$ de $\mathcal{M}$ est égale au fibré conormal de $\Supp(\mathcal{M})$ dans $X$. \\ \indent
Pour $\mathcal{M}\in D^{b}_{\hol}(\mathcal{D}_X)$, il passe suivant \cite[6.1-4]{Mehbsmf} par tout point de $X$ une hypersurface $Z$ telle que  
\begin{enumerate}
\item La trace de $Z$ sur le support de $\mathcal{M}$ est de codimension $1$ dans ce support\footnote{Par dimension d'un sous-espace analytique, il faut entendre la dimension maximale des composantes irréductibles de cet espace.}.
\item Les modules de cohomologie $\mathcal{H}^{k}\mathcal{M}$ sont lisses en dehors de $Z$.
\item 
La dimension du support de $R\Gamma_{[Z]}\mathcal{M}$ est strictement plus petite que la dimension du support de $\mathcal{M}$.
\end{enumerate}
\subsection{}\label{critèrefondamental}
Soit $\mathcal{M}\in D^{b}_{\hol}(\mathcal{D}_X)$ et soit $Z$ une hypersurface de $X$ satisfaisant aux points $(1)$ et $(2)$ de \ref{lemmemeb}. Alors, on a suivant \cite[4.3-16]{Mehbsmf} le\footnote{Stricto sensu, Mebkhout démontre \ref{soustheomeb} pour $\mathcal{M}$ module holonome, mais la généralisation au cas d'un complexe est aisée. Poser $\mathcal{N}:=\bigoplus_k \mathcal{H}^{k}\mathcal{M}$. On a par définition $\Supp \mathcal{N}=\Supp \mathcal{M}$. Par hypothèse, le support de $\Irr^{\ast}_{Z}(\mathcal{M})$ est un fermé de $Z$ de dimension $< \dim \Supp(\mathcal{M})-1$. D'après \cite[4.2-4]{Mehbsmf}, la nullité de $\Irr^{\ast}_{Z}(\mathcal{M})$ sur l'ouvert complémentaire implique la nullité de $\Irr^{\ast}_{Z}(\mathcal{N})$ sur ce même ouvert. On conclut alors en appliquant le critère fondamental de la régularité à $\mathcal{N}$ puis en utilisant la suite spectrale d'hypercohomologie 
$$
E_{2}^{pq}=\mathcal{H}^{p}\Irr^{\ast}_{Z}(\mathcal{H}^{-q}\mathcal{M})  \Longrightarrow \mathcal{H}^{p+q}\Irr^{\ast}_{Z}(\mathcal{M})
$$}
\begin{soustheorem}\label{soustheomeb}
Le complexe $\Irr^{\ast}_{Z}(\mathcal{M})$ est nul si et seulement si la dimension de son support est strictement plus petite que $ \dim \Supp(\mathcal{M})-1$.
\end{soustheorem}
Mebkhout en déduit le corollaire important
\begin{corollaire}\label{corrcriterfond}
Dans les conditions de \ref{soustheomeb}, le module $\mathcal{M}(\ast Z)$ est régulier.
\end{corollaire}
\subsection{}\label{dcomp} Pour tout $\mathcal{D}_{\mathds{C},0}$-module holonome $\mathcal{M}$, on
notera $\hat{\mathcal{M}}:=\hat{\mathcal{O}}_{\mathds{C},0}\otimes_{\mathcal{O}_{\mathds{C},0}}\mathcal{M}$. D'après \cite[6.3.1]{Courssimpa90}, on dispose d'une décomposition canonique
\begin{equation}\label{deomp}
\hat{\mathcal{M}}=\hat{\mathcal{M}}^{\reg}\oplus \hat{\mathcal{M}}^{\irr}
\end{equation}
avec $\hat{\mathcal{M}}^{\reg}$ régulier et $\hat{\mathcal{M}}^{\irr}$ satisfaisant à $\hat{\mathcal{M}}^{\irr}\simeq \hat{\mathcal{M}}^{\irr}(\ast 0)$.
\begin{lemme}\label{majorationn}
Supposons que $\mathcal{M}$ n'a pas de sections à support $0$, et posons $\mathcal{N}:=\mathcal{H}^{1}R\Gamma_{[0]}\mathcal{M}$. Alors on a
$$
\dim \mathbf{S}^{1}(\mathcal{N})_{0}\leq \rg \hat{\mathcal{M}}(\ast 0)^{\reg}
$$
\end{lemme}
\begin{proof}
On a la chaine d'isomorphismes
\begin{align}
\mathbf{S}(\mathcal{N})_{0}&\label{1}\simeq 
\mathbf{S}(\mathcal{N},\hat{\mathcal{O}})_{0}
 \simeq \mathbf{S}(\hat{\mathcal{N}},\hat{\mathcal{O}})_{0} \simeq \mathbf{S}(\hat{\mathcal{M}},\hat{\mathcal{O}})_{0}[-1]      \\
 &\label{2}\simeq \mathbf{S}(\hat{\mathcal{M}}^{\reg},\hat{\mathcal{O}})_{0}[-1]     \oplus \mathbf{S}(\hat{\mathcal{M}}^{\irr},\hat{\mathcal{O}})_{0}[-1]     \\
 &\label{3}\simeq\mathbf{S}(\hat{\mathcal{M}}^{\reg},\hat{\mathcal{O}})_{0}[-1] \\
 &\label{4}\simeq \mathbf{S}(\hat{\mathcal{M}}^{\reg}(\ast 0)/\hat{\mathcal{M}}^{\reg},\hat{\mathcal{O}})_{0}
\end{align}
où \eqref{1} provient du fait que $\mathcal{N}$ est régulier et de ce que le complexe des solutions formelles d'un module localisé est toujours nul\footnote{Voir la preuve de \cite[1.3.10]{Courssimpa90}}, où \eqref{2} provient de \eqref{deomp}. Enfin, \eqref{3} et \eqref{4} proviennent de ce que $\hat{\mathcal{M}}^{\irr}$ et $\hat{\mathcal{M}}^{\reg}(\ast 0)$ sont localisés.\\ \indent
Puisque le module 
$\hat{\mathcal{M}}^{\reg}(\ast 0)/\hat{\mathcal{M}}^{\reg}$ est à support $0$, il s'agit d'une somme de $m$ copies du Dirac en $0$. L'identification 
$$
\mathbf{S}(\mathcal{N})_{0}\simeq \mathbf{S}(\hat{\mathcal{M}}^{\reg}(\ast 0)/\hat{\mathcal{M}}^{\reg},\hat{\mathcal{O}})_{0}
$$
donne donc $m=\dim \mathbf{S}^{1}(\mathcal{N})_{0}$. En particulier, l'entier $\dim \mathbf{S}^{1}(\mathcal{N})_{0}$ est plus petit que le cardinal de n'importe quelle famille $\hat{\mathcal{D}}_{\mathds{C},0}$-génératrice de $\hat{\mathcal{M}}^{\reg}(\ast 0)$.\\ \indent
Pour montrer \ref{majorationn}, on est donc ramené à montrer qu'un $\mathds{C}((x))$-module différentiel régulier $M$ admet une famille $\hat{\mathcal{D}}_{\mathds{C},0}$-génératrice de cardinal égal à son rang, noté $\rg(M)$. Si $(e_1,\dots , e_{\rg(M)})$ est une base de $M$ sur laquelle $x\partial_{x}$ agit via une matrice triangulaire supérieure à coefficients complexes, alors pour un entier $n$ assez grand, la famille des $e_i/x^{n}, i=1,\dots ,\rg(M)$ convient.
\end{proof}

\section{Réduction du théorème \ref{theo1} au théorème \ref{theo2prime}}\label{Rappel}
Le contenu de cette section est entièrement dû à Mebkhout  \cite[7.3]{Mehbsmf} et ne prétend à aucune originalité. On adopte les notations de l'introduction.\\ \indent
L'observation principale de Mebkhout est la possibilité de factoriser le morphisme \eqref{RH} de la façon suivante
$$
\xymatrix{ R\mathcal{H}om(\mathcal{M}_1,\mathcal{M}_2) \ar[r]^{\sim}_-{(1)} \ar[d]_{RH_{\mathcal{M}_1,\mathcal{M}_2}}& \DR(\mathcal{M}_1^{\vee}\otimes^{\mathds{L}}_{\mathcal{O}_X}\mathcal{M}_2)   
      \ar[r]^-{\sim}_-{(2)}  & \mathbf{S}(\mathcal{M}_1^{\vee}\otimes^{\mathds{L}}_{\mathcal{O}_X}\mathcal{M}_2)^{\vee}\ar[d]^-{(3)}\\
       R\mathcal{H}om(\mathbf{S}(\mathcal{M}_2),\mathbf{S}(\mathcal{M}_1))     & \ar[l]_-{\sim}^-{(5)} R\mathcal{H}om(\mathbf{S}(\mathcal{M}_2),\mathbf{S}(\mathcal{M}_1^{\vee})^{\vee}   )   & \ar[l]_-{\sim}^-{(4)} (\mathbf{S}(\mathcal{M}_1^{\vee})\otimes\mathbf{S}(\mathcal{M}_2))^{\vee}        
}
$$
Les isomorphismes $(1)$ et $(4)$ sont de nature purement formelle. Les isomorphismes $(2)$ et $(5)$ proviennent du théorème de dualité locale \cite{MN}. Quant au morphisme $(3)$, il s'agit du morphisme dual du morphisme $(3)^{\vee}$ défini par la commutativité du diagramme
$$   
\xymatrix@C=6pc{
i_{\Delta_X}^{-1}\mathbf{S}(\mathcal{M}^{\vee}_1\boxtimes \mathcal{M}_2)\ar[r]^-{\CK_{i_{\Delta_X},\mathcal{M}^{\vee}_1\boxtimes \mathcal{M}_2}} \ar[d]^-\wr&   \mathbf{S}(i_{\Delta_X}^{+}(\mathcal{M}^{\vee}_1\boxtimes \mathcal{M}_2)) \ar[d]^-\wr \\
\mathbf{S}(\mathcal{M}^{\vee}_1)\otimes \mathbf{S}(\mathcal{M}_2) \ar[r]_{(3)^{\vee}}      &   \mathbf{S}(\mathcal{M}^{\vee}_1 \otimes^{\mathds{L}}_{\mathcal{O}_{X}} \mathcal{M}_2)
}
$$
où l'identification de gauche provient de  \ref{commutationsolbox}. \\ \indent
En particulier, le morphisme de Riemann-Hilbert $RH_{\mathcal{M}_1,\mathcal{M}_2}$ est un isomorphisme si et seulement si le morphisme $\CK_{i_{\Delta_X},\mathcal{M}^{\vee}_1\boxtimes \mathcal{M}_2}$
 est un isomorphisme. D'après la discussion \ref{CK},  le morphisme $RH_{\mathcal{M}_1,\mathcal{M}_2}$ est donc un isomorphisme si et seulement si $\Irr^{\ast}_{\Delta_X}(\mathcal{M}^{\vee}_1\boxtimes \mathcal{M}_2)$ est nul, et on a la réduction promise. 


\section{Quelques propriétés de l'irrégularité diagonale}\label{irrdiag}
\subsection{Irrégularités diagonale et ponctuelle}
Soit $X$ une variété complexe lisse et soient $\mathcal{M}_1,\mathcal{M}_2\in  D^{b}_{\hol}(\mathcal{D}_X)$. On se donne un point $x\in X$. On dispose du triangle de cohomologie locale
$$
R\Gamma_{[\Delta_X]}R(\mathcal{M}_{1}\boxtimes \mathcal{M}_{2})(\ast (x,x))\longrightarrow    R(\mathcal{M}_{1}\boxtimes \mathcal{M}_{2})(\ast (x,x))   \longrightarrow    R(\mathcal{M}_{1}\boxtimes \mathcal{M}_{2})(\ast \Delta_X) \longrightarrow 
$$
Or on a par \eqref{calculcohlocale}
\begin{align*}
R\Gamma_{[\Delta_X]}R(\mathcal{M}_{1}\boxtimes \mathcal{M}_{2})(\ast (x,x))&\simeq i_{\Delta_X +} i_{\Delta_X}^{+}R(\mathcal{M}_{1}\boxtimes \mathcal{M}_{2})(\ast (x,x))[-\dim X]\\ 
                  &\simeq  i_{\Delta_X +}R(i_{\Delta_X}^{+}(\mathcal{M}_{1}\boxtimes \mathcal{M}_{2}))(\ast x)[-\dim X] \\ 
                  &\simeq i_{\Delta_X +}R(\mathcal{M}_{1}\otimes^{\mathds{L}}_{\mathcal{O}_{X}}\mathcal{M}_{2})(\ast x)[-\dim X]
\end{align*}
Par compatibilité du foncteur solution avec l'image directe  \cite[4.2.5]{HTT}, on en déduit\footnote{Prendre garde aux décalages qui s'introduisent dans \cite[4.2.5]{HTT} pour le foncteur des solutions du fait de \cite[4.2.1]{HTT}. Utiliser aussi la formule $\mathbf{S}(\mathcal{M}[-\dim X])=\mathbf{S}(\mathcal{M})[\dim X]$. } une identification canonique dans $D^{b}_{c}(X,\mathds{C})$
$$
\mathbf{S}(R\Gamma_{[\Delta_X]}R(\mathcal{M}_{1}\boxtimes \mathcal{M}_{2})(\ast (x,x)))\simeq i_{\Delta_X \ast}\mathbf{S}(R(\mathcal{M}_{1}\otimes^{\mathds{L}}_{\mathcal{O}_{X}}\mathcal{M}_{2})(\ast x))
$$
Par application du foncteur solution au triangle de cohomologie locale précédent puis par prise des germes en $(x,x)$, on obtient un  triangle distingué de complexes d'espaces vectoriels à cohomologie bornée et de dimension finie
\begin{equation}\label{irrdiagetponct}
\xymatrix{
\Irr^{\ast}_{\Delta_{X}}(\mathcal{M}_{1}\boxtimes \mathcal{M}_{2})_{(x,x)} \ar[r] &  \Irr^{\ast}_{(x,x)}(\mathcal{M}_{1}\boxtimes \mathcal{M}_{2}) \ar[r]&\Irr^{\ast}_{x}(\mathcal{M}_{1}\otimes^{\mathds{L}}_{\mathcal{O}_{X}}\mathcal{M}_{2})  \ar[r]^-{+1}&
}
\end{equation}
\subsection{Le complexe $\Irr^{\ast}_{(x_1,x_2)}(\mathcal{M}_{1}\boxtimes \mathcal{M}_{2})$.} Dans ce paragraphe, on se donne des variétés complexes lisses $X_1$ et $X_2$, ainsi que des points $x_1\in X_1$ et $x_2\in X_2$. Enfin pour $i=1,2$, on se donne un complexe $\mathcal{M}_{i}\in D^{b}_{\hol}(\mathcal{D}_{X_i})$. \\ \indent
D'après \ref{compproduittens}, on a
$$
R\Gamma_{[(x_1,x_2)]}(\mathcal{M}_{1}\boxtimes \mathcal{M}_{2})\simeq R\Gamma_{[x_1]}\mathcal{M}_{1}\boxtimes R\Gamma_{[x_2]}\mathcal{M}_{2}
$$
Par application du foncteur solution au triangle de cohomologie locale
$$
\xymatrix{
R\Gamma_{[(x_1,x_2)]}(\mathcal{M}_{1}\boxtimes \mathcal{M}_{2})\ar[r] &  \mathcal{M}_{1}\boxtimes \mathcal{M}_{2}\ar[r]&  R(\mathcal{M}_{1}\boxtimes \mathcal{M}_{2})(\ast (x_1,x_2))\ar[r]^-{+1}&
}
$$
le point \ref{commutationsolbox} donne un triangle distingué de  $D^{b}_{c}(x_1\times x_2,\mathds{C})$
\begin{equation}\label{triangle2}
\Irr^{\ast}_{(x_1,x_2)}(\mathcal{M}_{1}\boxtimes \mathcal{M}_{2}) \longrightarrow  \mathbf{S}(\mathcal{M}_{1})_{x_1}\otimes \mathbf{S}(\mathcal{M}_{2})_{x_2}\longrightarrow \mathbf{S}(R\Gamma_{[x_1]}\mathcal{M}_{1})_{x_1}\otimes  \mathbf{S}(R\Gamma_{[x_2]}\mathcal{M}_{2})_{x_2}\overset{+1}{\longrightarrow }
\end{equation}

\subsection{Compatibilité avec la restriction non-caractéristique}\label{restrictioncar}
Soit $X$ une variété complexe lisse et soit $\mathcal{M}\in D^{b}_{\hol}(\mathcal{D}_X)$. On se donne une sous-variété lisse $X^{\prime}$ de $X$, ainsi qu'une sous-variété lisse $Z$ de $X$ telle que $Z^{\prime}:=Z\cap X^{\prime}$ est lisse. On a en particulier un diagramme cartésien 
\begin{equation}\label{diagrammecart}
\xymatrix{
Z^{\prime} \ar[r] \ar[d]_-{i_{Z^{\prime}}^{\prime}} &  Z \ar[d]^-{i_{Z}} \\
X^{\prime} \ar[r]_{i_{X^{\prime} }}      &   X
}
\end{equation}
Montrons la 
\begin{proposition}\label{restrcar}
Si $X^{\prime}$ est non-caractéristique pour $\mathcal{M}$, et si $Z^{\prime}$ est non caractéristique pour $i_Z^{+}\mathcal{M}$, alors on a une identification canonique
$$
\Irr_{Z^{\prime}}^{\ast}(i^{+}_{X^{\prime}}\mathcal{M})\overset{\sim}{\longrightarrow} i_{Z^{\prime}\ast}^{\prime}i^{-1}_{Z^{\prime}}\Irr^{\ast}_{Z}(\mathcal{M})
$$
\end{proposition}
\begin{proof}
Par application de $i_{Z^{\prime}\ast}^{\prime}i_{Z^{\prime}}^{\prime-1}\mathbf{S}$
à l'isomorphisme de triangles \eqref{cohlociminv}, on obtient un isomorphisme de triangles
$$
\xymatrix{
 i_{Z^{\prime}\ast}^{\prime}i_{Z^{\prime}}^{\prime-1}\mathbf{S}(i^{+}_{X^{\prime}}R\mathcal{M}(\ast Z)) \ar[r] \ar[d]^-\wr&   i_{Z^{\prime}\ast}^{\prime}i_{Z^{\prime}}^{\prime-1}\mathbf{S}(i^{+}_{X^{\prime}}\mathcal{M}  ) \ar[r] \ar[d]^-\wr&    i_{Z^{\prime}\ast}^{\prime}i_{Z^{\prime}}^{\prime-1}\mathbf{S}(i^{+}_{X^{\prime}}R\Gamma_{[Z]}\mathcal{M}) \ar[d]^-\wr \\
  \Irr^{\ast}_{Z^{\prime}}(i^{+}_{X^{\prime}}\mathcal{M})  \ar[r] &     i_{Z^{\prime}\ast}^{\prime}i_{Z^{\prime}}^{\prime-1}\mathbf{S}( i^{+}_{X^{\prime}}\mathcal{M} )       \ar[r] &   i_{Z^{\prime}\ast}^{\prime}i_{Z^{\prime}}^{\prime-1}\mathbf{S}(R\Gamma_{[Z^{\prime}]}i^{+}_{X^{\prime}}\mathcal{M} )
}
$$
Or le morphisme de Cauchy-Kowalevski induit par fonctorialité un morphisme de triangles
$$
\xymatrix{
i^{-1}_{X^{\prime}}\mathbf{S}(R\mathcal{M}(\ast Z))\ar[r] \ar[d]^{\CK_{X^{\prime},R\mathcal{M}(\ast Z)}}&    i^{-1}_{X^{\prime}} \mathbf{S}(\mathcal{M})    \ar[r] \ar[d]^{\CK_{X^{\prime},\mathcal{M}}}&      i^{-1}_{X^{\prime}}\mathbf{S}(R\Gamma_{[Z]}(\mathcal{M}))   \ar[d]^{\CK_{X^{\prime},R\Gamma_{[Z]}(\mathcal{M})}} \\
\mathbf{S}(i^{+}_{X^{\prime}}R\mathcal{M}(\ast Z))\ar[r] &    \mathbf{S}(i^{+}_{X^{\prime}}\mathcal{M})    \ar[r] &               
 \mathbf{S}(i^{+}_{X^{\prime}}R\Gamma_{[Z]}(\mathcal{M}))
}
$$
Par application de $i_{Z^{\prime}\ast}^{\prime}i_{Z^{\prime}}^{\prime-1}$, 
on en déduit un morphisme de triangles
\begin{equation}\label{motriangle}
\xymatrix{
i_{Z^{\prime}\ast}^{\prime}i^{-1}_{Z^{\prime}}\mathbf{S}(R\mathcal{M}(\ast Z))\ar[r] \ar[d]&       i_{Z^{\prime}\ast}^{\prime}i^{-1}_{Z^{\prime}}\mathbf{S}(\mathcal{M})  \ar[r] \ar[d]&         i_{Z^{\prime}\ast}^{\prime}i^{-1}_{Z^{\prime}}\mathbf{S}(R\Gamma_{[Z]}(\mathcal{M}))\ar[d]\\
i_{Z^{\prime}\ast}^{\prime}i_{Z^{\prime}}^{\prime-1}\mathbf{S}(i^{+}_{X^{\prime}}R\mathcal{M}(\ast Z))\ar[r] &          i_{Z^{\prime}\ast}^{\prime}i_{Z^{\prime}}^{\prime-1}\mathbf{S}(i^{+}_{X^{\prime}}\mathcal{M}) \ar[r] &  
i_{Z^{\prime}\ast}^{\prime}i_{Z^{\prime}}^{\prime-1}\mathbf{S}(i^{+}_{X^{\prime}}R\Gamma_{[Z]}(\mathcal{M}))              
}
\end{equation}
Or
$$
i_{Z^{\prime}\ast}^{\prime}i^{-1}_{Z^{\prime}}\mathbf{S}(R\mathcal{M}(\ast Z))\simeq i_{Z^{\prime}\ast}^{\prime}i^{-1}_{Z^{\prime}}i_{Z*}i^{-1}_{Z}\mathbf{S}(R\mathcal{M}(\ast Z))=i_{Z^{\prime}\ast}^{\prime}i^{-1}_{Z^{\prime}}\Irr_{Z}^{\ast}(\mathcal{M})
$$
où la première identification provient simplement de $Z^{\prime}\subset Z$.\\ \indent
On est donc ramené à montrer que le morphisme vertical le plus à gauche de \eqref{motriangle} est un isomorphisme. Il suffit pour cela de voir que $\CK_{X^{\prime},\mathcal{M}}$ et $\CK_{X^{\prime},R\Gamma_{[Z]}(\mathcal{M})}$ sont des isomorphismes.\\ \indent
Pour $\CK_{X^{\prime},\mathcal{M}}$, cela provient de l'hypothèse que $X^{\prime}$ est non caractéristique pour $\mathcal{M}$. D'après \ref{CK}, on sait que l'obstruction à ce que $\CK_{X^{\prime},R\Gamma_{[Z]}(\mathcal{M})}$ soit un isomorphisme réside dans la non-nullité de 
$\Irr^{\ast}_{X^{\prime}}(R\Gamma_{[Z]}(\mathcal{M}))$. Or d'après \eqref{calculcohlocale}, on a 
$$
R\Gamma_{[Z]}(\mathcal{M})\simeq i_{Z + }i^{+}_{Z}\mathcal{M}[-\codim_{X}Z]
$$
Par compatibilité de l'irrégularité avec les morphismes propres \cite[3.6-6]{Mehbsmf}, on en déduit
\begin{equation}\label{compirrpropre}
\Irr^{\ast}_{X^{\prime}}(R\Gamma_{[Z]}(\mathcal{M}))\simeq i_{Z \ast}\Irr^{\ast}_{Z^{\prime}}(i^{+}_{Z}\mathcal{M})
\end{equation}
Le membre de droite de \eqref{compirrpropre} est nul du fait de l'hypothèse que $Z^{\prime}$ est non caractéristique pour $i^{+}_{Z}\mathcal{M}$. D'où la proposition \ref{restrcar}.
\end{proof}
On en déduit le 
\begin{corollaire}\label{restricara}
Si $\mathcal{M}_1,\mathcal{M}_2\in D^{b}_{\hol}(\mathcal{D}_X)$, et si $X^{\prime}$ désigne une sous-variété lisse de $X$ non-caractéristique pour $\mathcal{M}_{1}$,  $\mathcal{M}_{2}$ et $\mathcal{M}_{1}\otimes^{\mathds{L}}_{\mathcal{O}_{X}}\mathcal{M}_{2}$, alors on a un isomorphisme canonique
$$
\Irr^{\ast}_{\Delta_{X^{\prime}}}(i^{+}_{X^{\prime}}\mathcal{M}_{1}\boxtimes i^{+}_{X^{\prime}}\mathcal{M}_{2})\overset{\sim}{\longrightarrow}i^{\prime}_{\Delta_{X^{\prime}}\ast}i^{-1}_{\Delta_{X^{\prime}}}\Irr^{\ast}_{\Delta_{X}}(\mathcal{M}_{1}\boxtimes \mathcal{M}_{2}) 
$$
\end{corollaire}
\begin{proof}
Du fait de la formule \cite[4.3]{MT}
$$
\car(\mathcal{M}_{1}\boxtimes \mathcal{M}_{2})=\car(\mathcal{M}_{1})\times \car(\mathcal{M}_{2})
$$
la sous-variété $X^{\prime}\times X^{\prime}$ de $X\times X$ est non-caractéristique pour $\mathcal{M}_{1}\boxtimes \mathcal{M}_{2}$. Or on a 
$$
i_{\Delta_X}^{+}(\mathcal{M}_1\boxtimes \mathcal{M}_2)\simeq \mathcal{M}_1\otimes^{\mathds{L}}_{\mathcal{O}_{X}}\mathcal{M}_2
$$
où l'on a implicitement identifié $\Delta_X$ et $X$. A travers cette identification, $\Delta_{X^{\prime}}=\Delta_{X}\cap (X^{\prime}\times X^{\prime})$ correspond à $X^{\prime}$. Par hypothèse, $\Delta_{X^{\prime}}$ est donc non-caractéristique pour $i_{\Delta_X}^{+}(\mathcal{M}_1\boxtimes \mathcal{M}_2)$. Il suffit alors d'appliquer \ref{restrcar}.
\end{proof}
\section{Amplitude de l'irrégularité diagonale dans le cas de modules à support une courbe }\label{amplitude}
\subsection{}\label{notations}
Dans tout ce qui suit, on notera $d$ la dimension de $X$. On se donne des $\mathcal{D}_X$-modules holonomes $\mathcal{M}_1,\mathcal{M}_2$ à supports contenus dans une courbe $C$ et supposés réguliers en dehors d'un point $x_0\in C$. On note $C_1, \dots, C_n$ les composantes irréductibles de $C$. \\ \indent
Soit $D$ une hypersurface de $X$ passant par $x_0$ et ne contenant aucune des $C_i$. Notons $\tilde{C}_1,\dots, \tilde{C}_n$ les composantes connexes de la normalisation $\tilde{C}$ de $C$. Chaque $\tilde{C}_i$ s'envoie bijectivement sur une composante irréductible de $C$. On note $x_i\in \tilde{C}_i$ le point de $\tilde{C}_i$ s'envoyant sur $x_0$. On a $\tilde{C}\times \tilde{C}=\bigsqcup \tilde{C}_i\times \tilde{C}_j$. On note $p:\tilde{C}\longrightarrow X$ la composée de la projection $\tilde{C}\longrightarrow C$ et de l'inclusion $C\hookrightarrow X$.\\ \indent
Puisque $p$ est propre, on sait d'après               \cite[3.6-4]{Mehbsmf}  que pour $i=1,2$ on a
$$
(p_{+}p^{\dag}\mathcal{M}_i)(\ast D)\simeq p_{+}\left[(p^{\dag}\mathcal{M}_i)(\ast p^{-1}(D))\right]\simeq 
p_{+}p^{\dag}(\mathcal{M}_i(\ast D))
$$
D'autre part, $p$ induit une immersion fermée en dehors de $p^{-1}(x_0)$. Par théorème de Kashiwara \cite[1.6.1]{HTT},  l'adjonction canonique \cite[7.1]{MS}
$$
p_{+}p^{\dag}\mathcal{M}_i\longrightarrow \mathcal{M}_i
$$
induit un isomorphisme en dehors de $x_0$. Son cône est donc à support $x_0$. D'après \ref{coholocalsupport}, on en déduit une identification 
\begin{equation}\label{identification}
p_{+}p^{\dag}(\mathcal{M}_i(\ast D))\simeq  \mathcal{M}_i(\ast D)
\end{equation}
Toujours par théorème de Kashiwara, on sait que pour $k\neq 0$, le module $\mathcal{H}^{k}p^{\dag}\mathcal{M}_i$ est à support dans $p^{-1}(x_0)$. On en déduit que  $p^{\dag}(\mathcal{M}_i(\ast D))\simeq (p^{\dag}\mathcal{M}_i)(\ast p^{-1}(x_0))$ est concentré en degré $0$. Autrement dit, il s'agit d'un module holonome. On note $\mathcal{M}_{i,j}(\ast x_j)$ la restriction de $p^{\dag}(\mathcal{M}_i(\ast D))$ à $\tilde{C}_j$ et $\irr_{i,j}$ le nombre d'irrégularité de $\mathcal{M}_{i,j}(\ast x_j)$ en $x_j$. On pose aussi
$$
\irr_i :=\displaystyle{\sum_{j=1}^{n}}   \irr_{i,j}  
$$
Notons enfin $\delta$ le module Dirac en $x_0$. Il s'agit d'un module régulier. Le but de cette section est de montrer la 
\begin{proposition}\label{lemme0}
Le complexe $\Irr^{\ast}_{\Delta_{X}}(\mathcal{M}_1\boxtimes  \mathcal{M}_2) $ est concentré en degrés $2d-1$ et  $2d$.
\end{proposition}
\subsection{Calcul de $\Irr^{\ast}_{\Delta_{X}}(\delta\boxtimes \mathcal{M}_{2})$}\label{refamettre}
En appliquant $\Irr^{\ast}_{\Delta_{X}}$ au triangle distingué
$$
0\longrightarrow \delta\boxtimes  R\Gamma_{[D]}\mathcal{M}_2\longrightarrow  \delta\boxtimes  \mathcal{M}_2 \longrightarrow   \delta\boxtimes  \mathcal{M}_2(\ast D) \longrightarrow 0
$$
on obtient une identification 
$$
\Irr^{\ast}_{\Delta_{X}}(\delta\boxtimes  \mathcal{M}_2)\simeq 
\Irr^{\ast}_{\Delta_{X}}(\delta\boxtimes  \mathcal{M}_2(\ast D))
$$
On est donc ramené à calculer $\Irr^{\ast}_{\Delta_{X}}(\delta\boxtimes  \mathcal{M}_2(\ast D))$.
Or puisque  $\delta\otimes^{\mathds{L}}_{\mathcal{O}_{X}} \mathcal{M}_{2}(\ast D)\simeq 0$, le triangle \eqref{irrdiagetponct} donne une identification 
$$
\Irr^{\ast}_{\Delta_{X}}(\delta\boxtimes \mathcal{M}_{2} (\ast D))_{(x_0,x_0)} \simeq  \Irr^{\ast}_{(x_0,x_0)}(\delta\boxtimes \mathcal{M}_{2} (\ast D)) 
$$
Puisque $D$ passe par $x_0$, on a
$$
R\Gamma_{[x_0]}(\mathcal{M}_{2} (\ast D))\simeq R\Gamma_{[x_0]}R\Gamma_{[D]}(\mathcal{M}_{2} (\ast D))\simeq 0
$$
Le triangle $(3.3.1)$ donne donc une identification 
$$
\Irr^{\ast}_{(x_0,x_0)}(\delta\boxtimes \mathcal{M}_{2} (\ast D)) \simeq \mathbf{S}(\delta)_{x_0}\otimes \mathbf{S}(\mathcal{M}_{2}(\ast D))_{x_0}
$$
Or par compatibilité du foncteur solution avec l'image directe propre \cite[4.2.5]{HTT}, l'identification \eqref{identification} donne 
\begin{align*}
\mathbf{S}(\mathcal{M}_{2}(\ast D))_{x_0}&\simeq
\mathbf{S}(p_{+}p^{\dag}(\mathcal{M}_2(\ast D)))_{x_0}\\
 & \simeq (p_{\ast}\mathbf{S}(p^{\dag}(\mathcal{M}_{2}(\ast D))))_{x_0}[1-d] \\ \indent
 &\simeq\displaystyle{\bigoplus_{j=1}^{n}}
\mathbf{S}(\mathcal{M}_{2,j}(\ast x_j))_{x_j}[1-d]
\end{align*}
On a ainsi obtenu le 
\begin{lemme}\label{calcul-1}
Le complexe $\Irr^{\ast}_{\Delta_{X}}(\delta\boxtimes \mathcal{M}_{2} )_{(x_0,x_0)} $ est concentré en degré $2d$ et on a 
$$
\dim \mathcal{H}^{2d}\Irr^{\ast}_{\Delta_{X}}(\delta\boxtimes \mathcal{M}_{2} )_{(x_0,x_0)} =\irr_2
$$
\end{lemme}

\subsection{Calcul de $\Irr^{\ast}_{\Delta_{X}}( \mathcal{M}_1 (\ast D)
\boxtimes \mathcal{M}_2 ( \ast D)) $}
D'après l' identification \eqref{identification} pour $i=1,2$, on a
\begin{align*}
\Irr^{\ast}_{\Delta_{X}}( \mathcal{M}_1 (\ast D)
\boxtimes \mathcal{M}_2 ( \ast D))& \simeq \Irr^{\ast}_{\Delta_{X}}( p_{+}p^{\dag}(\mathcal{M}_1 (\ast D)
)\boxtimes p_{+}p^{\dag}(\mathcal{M}_2 ( \ast D))) \\
&\simeq 
\Irr^{\ast}_{\Delta_{X}}\left( (p\times p)_{+}(p^{\dag}(\mathcal{M}_1 (\ast D)
)\boxtimes p^{\dag}(\mathcal{M}_2 ( \ast D)))\right)
\end{align*}
où le second isomorphisme provient de \cite[1.5.30]{HTT}.
Par compatibilité \ref{compmorphismepropre} de l'irrégularité avec l'image directe propre, il vient
$$
\Irr^{\ast}_{\Delta_{X}}( \mathcal{M}_1 (\ast D)
\boxtimes \mathcal{M}_2 ( \ast D))\simeq (p\times p)_{\ast}\Irr^{\ast}_{\Delta_{X}^{\prime}}(p^{\dag}(\mathcal{M}_1 (\ast D)
)\boxtimes p^{\dag}(\mathcal{M}_2 (\ast D)
))[2-2d]$$
où on a noté $\Delta_{X}^{\prime}$ pour $(p\times p)^{-1}(\Delta_{X})$. Donc  $\Irr^{\ast}_{\Delta_{X}}( \mathcal{M}_1 (\ast D)
\boxtimes \mathcal{M}_2 ( \ast D))_{(x_0,x_0)}$ s'identifie à 
$$
\bigoplus_{(i,j)}
\Irr^{\ast}_{\Delta_{X}^{\prime}\cap (\tilde{C}_i\times \tilde{C}_j)}(\mathcal{M}_{1,i} (\ast x_i)\boxtimes \mathcal{M}_{2,j} (\ast x_j))_{(x_i,x_j)}[2-2d]
$$
Il y a deux cas possibles:
\begin{enumerate}
\item si $i\neq j$, on a $(\tilde{C}_i\times \tilde{C}_j)\cap \Delta_{X}^{\prime}=(x_i,x_j)$.
\item si $i=j$, on a $(\tilde{C}_i\times \tilde{C}_i)\cap \Delta_{X}^{\prime}=\Delta_{\tilde{C}_i}$
\end{enumerate}
Dans le premier cas, le triangle \eqref{triangle2} donne 
\begin{equation}\label{degre2}
\Irr^{\ast}_{(x_i,x_j)}(\mathcal{M}_{1,i} (\ast x_i)\boxtimes \mathcal{M}_{2,j}(\ast x_j))_{(x_i,x_j)}\simeq \mathbf{S}(\mathcal{M}_{1,i} (\ast x_i))_{x_i}\otimes \mathbf{S}(\mathcal{M}_{2,j} (\ast x_j))_{x_j}
\end{equation}
En particulier, le membre de gauche de \eqref{degre2} est concentré en degré $2$ et on a 
$$
\dim \mathcal{H}^{2}\Irr^{\ast}_{(x_i,x_j)}(\mathcal{M}_{1,i} (\ast x_i)\boxtimes \mathcal{M}_{2,j}(\ast x_j))_{(x_i,x_j)}=\irr_{1,i}\irr_{2,j}
$$
Dans le second cas, on observe que la diagonale $\Delta_{\tilde{C}_i}$  est une hypersurface de $\tilde{C}_i\times \tilde{C}_i$. Par théorème de perversité de Mebkhout \cite[2.1.6]{Mehbgro}, on en déduit que le complexe $$\Irr^{\ast}_{\Delta_{\tilde{C}_i}}(\mathcal{M}_{1,i}(\ast x_i)\boxtimes \mathcal{M}_{2,i}(\ast x_i))$$ est pervers. Puisqu'il est nul en dehors de $(x_i,x_i)$, il s'agit d'un faisceau gratte-ciel concentré en degré $2$. Le triangle 
\eqref{irrdiagetponct} donne
$$
\dim \mathcal{H}^{2}\Irr^{\ast}_{\Delta_{\tilde{C}_i}}(\mathcal{M}_{1,i} (\ast x_i)\boxtimes \mathcal{M}_{2,i} (\ast x_i))=\irr_{1,i}\irr_{2,i}+\irr_{x_i}(\mathcal{M}_{1,i}(\ast x_i)\otimes_{\mathcal{O}_{\tilde{C}_i}}  \mathcal{M}_{2,i}(\ast x_i))
$$
On en déduit le 
\begin{lemme}\label{calcul2}
Le complexe $\Irr^{\ast}_{\Delta_{X}}( \mathcal{M}_1 (\ast D)
\boxtimes \mathcal{M}_2 ( \ast D)) $ est concentré en degré $2d$, et on a 
$$\dim \mathcal{H}^{2d}\Irr^{\ast}_{\Delta_{X}}( \mathcal{M}_1 (\ast D)
\boxtimes \mathcal{M}_2 ( \ast D)) = \irr_{1}\irr_{2}+\displaystyle{\sum_{i=1}^{n}}\irr_{x_i}(\mathcal{M}_{1,i}(\ast x_i)\otimes_{\mathcal{O}_{\tilde{C}_i}}  \mathcal{M}_{2,i}(\ast x_i))$$
\end{lemme}

\subsection{Preuve de \ref{lemme0}}\label{amplitudebis}
Notons $\mathcal{C}_i:=R\Gamma_{[D]}\mathcal{M}_i$. Il s'agit d'un complexe n'ayant que du $\mathcal{H}^{0}$ et du $\mathcal{H}^{1}$, avec $\mathcal{H}^{0}\mathcal{C}_i$ et $\mathcal{H}^{1}\mathcal{C}_i$  à support $x_0$. Les triangles de cohomologie locale
\begin{equation}\label{tricohloc}
\xymatrix{
\mathcal{C}_i\ar[r] &  \mathcal{M}_i\ar[r]&  \mathcal{M}_i (\ast D)\ar[r]^-{+1}&
}
\end{equation}
pour $i=1,2$ fournissent un diagramme
\begin{equation}\label{gros diagramme}
\xymatrix{
\mathcal{C}_1\boxtimes \mathcal{C}_2\ar[r] \ar[d]&           \mathcal{M}_1\boxtimes  \mathcal{C}_2 \ar[r] \ar[d]&                \mathcal{M}_1 (\ast D)\boxtimes \mathcal{C}_2 \ar[d] \\
\mathcal{C}_1
\boxtimes \mathcal{M}_2\ar[r] \ar[d]&          \mathcal{M}_1
\boxtimes \mathcal{M}_2\ar[r] \ar[d]&              \mathcal{M}_1 (\ast D)
\boxtimes\mathcal{M}_2  \ar[d] \\
\mathcal{C}_1
\boxtimes  \mathcal{M}_2 (\ast D) \ar[r] &          \mathcal{M}_1
\boxtimes \mathcal{M}_2 (\ast D)
\ar[r] &               \mathcal{M}_1 (\ast D)
\boxtimes \mathcal{M}_2 ( \ast D)
}
\end{equation}
à lignes et colonnes distinguées. Par application de $\Irr^{\ast}_{\Delta_{X}}$, on obtient un digramme 
\begin{equation}\label{tresgrosdiagramme}
\xymatrix{
 0&           \Irr^{\ast}_{\Delta_{X}}(\mathcal{M}_1\boxtimes  \mathcal{C}_2 )\ar[l] &                \Irr^{\ast}_{\Delta_{X}}(\mathcal{M}_1 (\ast D)\boxtimes \mathcal{C}_2) \ar[l]_-{\sim}  \\
\Irr^{\ast}_{\Delta_{X}}(     \mathcal{C}_1
\boxtimes \mathcal{M}_2 ) \ar[u]&   \Irr^{\ast}_{\Delta_{X}}(       \mathcal{M}_1
\boxtimes \mathcal{M}_2)\ar[l] \ar[u]&        \Irr^{\ast}_{\Delta_{X}}( \mathcal{M}_1 (\ast D)
\boxtimes\mathcal{M}_2)\ar[u]\ar[l]  \\
 \Irr^{\ast}_{\Delta_{X}}(\mathcal{C}_1
\boxtimes  \mathcal{M}_2 (\ast D)       )  \ar[u]^-{\wr}  &    \ar[u] \Irr^{\ast}_{\Delta_{X}}(   \mathcal{M}_1
\boxtimes \mathcal{M}_2 (\ast D))
\ar[l] \ar[u]&   \Irr^{\ast}_{\Delta_{X}}( \mathcal{M}_1 (\ast D)
\boxtimes \mathcal{M}_2 ( \ast D)) \ar[u]\ar[l] 
}
\end{equation}
Pour démontrer la proposition \ref{lemme0}, il suffit donc de démontrer que  $\Irr^{\ast}_{\Delta_{X}}(     \mathcal{C}_1
\boxtimes \mathcal{M}_2 )$ et  $\Irr^{\ast}_{\Delta_{X}}( \mathcal{M}_1 (\ast D)
\boxtimes\mathcal{M}_2)$ sont concentrés en degrés $2d-1$ et  $2d$. \\ \indent
Traitons le cas du complexe $\Irr^{\ast}_{\Delta_{X}}(     \mathcal{C}_1
\boxtimes \mathcal{M}_2 )$. D'après \ref{calcul-1}, la suite spectrale d'hypercohomologie 
\begin{equation}\label{ssdeg}
E_{2}^{pq}=\mathcal{H}^{p}\Irr^{\ast}_{\Delta_{X}}((\mathcal{H}^{-q}\mathcal{C}_1)
\boxtimes \mathcal{M}_2)  \Longrightarrow \mathcal{H}^{p+q}\Irr^{\ast}_{\Delta_{X}}(\mathcal{C}_1
\boxtimes \mathcal{M}_2)
\end{equation}
n'a que les termes $E_{2}^{2d,-1}$ et $E_{2}^{2d,0}$ pour termes éventuellement non nuls. La suite spectrale \eqref{ssdeg} dégénère donc en page $2$, et on en déduit que $\Irr^{\ast}_{\Delta_{X}}(\mathcal{C}_1
\boxtimes  \mathcal{M}_2 )$ est concentré en degrés $2d-1$ et $2d$. \\ \indent 
Pour le complexe $\Irr^{\ast}_{\Delta_{X}}( \mathcal{M}_1 (\ast D)
\boxtimes\mathcal{M}_2)$, on observe en contemplant  \eqref{tresgrosdiagramme} qu'il suffit de montrer que  $ \Irr^{\ast}_{\Delta_{X}}(\mathcal{M}_1 (\ast D)\boxtimes \mathcal{C}_2)$ et $\Irr^{\ast}_{\Delta_{X}}( \mathcal{M}_1 (\ast D)
\boxtimes \mathcal{M}_2 ( \ast D))$ sont concentrés en degrés $2d-1$ et $2d$. Pour $ \Irr^{\ast}_{\Delta_{X}}(\mathcal{M}_1 (\ast D)\boxtimes \mathcal{C}_2)$, il s'agit par symétrie du même argument que pour $\Irr^{\ast}_{\Delta_{X}}(\mathcal{C}_1
\boxtimes \mathcal{M}_2 )$. Pour $\Irr^{\ast}_{\Delta_{X}}( \mathcal{M}_1 (\ast D)
\boxtimes \mathcal{M}_2 ( \ast D))$, cela découle de \ref{calcul2}. \\ \indent

Notons $\mathcal{N}_i:=\mathcal{H}^{0}R\Gamma_{[D]}\mathcal{M}_i$. C'est un sous-module de $\mathcal{M}_i$ à support $x_0$. On a le
\begin{corollaire}\label{petite réduction}
Dans les conditions de \ref{notations}, la nullité de $\Irr^{\ast}_{\Delta_{X}}(\mathcal{M}_1\boxtimes  \mathcal{M}_2)$ entraine la nullité de $\Irr^{\ast}_{\Delta_{X}}(\mathcal{M}_1/\mathcal{N}_1\boxtimes  \mathcal{M}_2)$.
\end{corollaire}
\begin{proof}
On a une suite exacte
$$
0\longrightarrow \mathcal{N}_1\longrightarrow  \mathcal{M}_1 \longrightarrow   \mathcal{M}_1/\mathcal{N}_1 \longrightarrow 0
$$
On en déduit une suite exacte 
$$
0\longrightarrow \mathcal{N}_1\boxtimes  \mathcal{M}_2\longrightarrow  \mathcal{M}_1\boxtimes  \mathcal{M}_2 \longrightarrow   \mathcal{M}_1/\mathcal{N}_1\boxtimes  \mathcal{M}_2 \longrightarrow 0
$$
D'où un triangle distingué
$$
\Irr^{\ast}_{\Delta_{X}}(\mathcal{M}_1/\mathcal{N}_1\boxtimes  \mathcal{M}_2)\longrightarrow  \Irr^{\ast}_{\Delta_{X}}(\mathcal{M}_1\boxtimes  \mathcal{M}_2) \longrightarrow   \Irr^{\ast}_{\Delta_{X}}(\mathcal{N}_1\boxtimes  \mathcal{M}_2) \overset{+1}{\longrightarrow} 
$$
Alors on a par hypothèse
$$
\Irr^{\ast}_{\Delta_{X}}(\mathcal{M}_1/\mathcal{N}_1\boxtimes  \mathcal{M}_2)\simeq \Irr^{\ast}_{\Delta_{X}}(\mathcal{N}_1\boxtimes  \mathcal{M}_2)[-1]
$$
D'après \ref{lemme0}, on sait que $\Irr^{\ast}_{\Delta_{X}}(\mathcal{M}_1/\mathcal{N}_1\boxtimes  \mathcal{M}_2)$ et $\Irr^{\ast}_{\Delta_{X}}(\mathcal{N}_1\boxtimes  \mathcal{M}_2)$ sont concentrés en degrés $2d-1$ et $2d$. On a donc
$$
\mathcal{H}^{2d-1}\Irr^{\ast}_{\Delta_{X}}(\mathcal{M}_1/\mathcal{N}_1\boxtimes  \mathcal{M}_2)\simeq \mathcal{H}^{2d-2}\Irr^{\ast}_{\Delta_{X}}(\mathcal{N}_1\boxtimes  \mathcal{M}_2)\simeq 0
$$
et 
$$
\mathcal{H}^{2d}\Irr^{\ast}_{\Delta_{X}}(\mathcal{M}_1/\mathcal{N}_1\boxtimes  \mathcal{M}_2)\simeq \mathcal{H}^{2d-1}\Irr^{\ast}_{\Delta_{X}}(\mathcal{N}_1\boxtimes  \mathcal{M}_2)
$$
Pour démontrer \ref{petite réduction}, il suffit donc de montrer que 
$\Irr^{\ast}_{\Delta_{X}}(\mathcal{N}_1\boxtimes  \mathcal{M}_2)$ est concentré en degré $2d$. C'est bien le cas d'après \ref{calcul-1}.
\end{proof}
\section{Preuve du théorème \ref{theo2prime}}
On va raisonner par récurrence sur la dimension $d$ de $X$. 
\subsection{Le cas d'une courbe}\label{cascourbe}
Pour les besoins de la preuve, on va démontrer dans le cas d'une courbe un énoncé sensiblement plus général que le théorème \ref{theo2prime}, à savoir la
\begin{proposition}\label{uneprop}
Soient $\mathcal{M}_1$ et $\mathcal{M}_2$ des modules holonomes définis sur un voisinage de $0\in \mathds{C}$ et supposés lisses en dehors de $0$. On suppose que 
$$\Irr^{\ast}_{\Delta_{X}}(\mathcal{M}_1
\boxtimes \mathcal{M}_2) \simeq 0$$
Alors $\mathcal{M}_1$ ou $\mathcal{M}_2$ est régulier.
\end{proposition}
\begin{proof}
D'après \ref{petite réduction}, on peut supposer que  $\mathcal{M}_1$  et $\mathcal{M}_2$ n'ont pas de sections à support $0$. 
Si pour $i=1,2$ on pose $\mathcal{N}_i= \mathcal{H}^{1}R\Gamma_{[D]}(\mathcal{M}_i)$, on a une suite exacte
\begin{equation}\label{usiste exactcohlocal}
0\longrightarrow \mathcal{M}_i\longrightarrow  \mathcal{M}_i(\ast D) \longrightarrow   \mathcal{N}_i \longrightarrow 0
\end{equation}
Le complexe $\mathbf{S}(\mathcal{N}_i)_{0}$ est concentré en degré $1$, et par perversité des solutions d'un module holonome, le complexe $\mathbf{S}(\mathcal{M}_i)_{0}$ est concentré en degrés $0$ et $1$. D'autre part, $\mathbf{S}(\mathcal{M}_i(\ast D))_{0}$ est concentré en degré $1$, et par définition $\dim \mathbf{S}^{1}(\mathcal{M}_i(\ast D))_{0}=\irr_i$. On posera $m_i=\dim \mathbf{S}^{0}(\mathcal{M}_{i})_{0}$, $\mu_i=\dim \mathbf{S}^{1}(\mathcal{M}_{i})_{0}$, $ n_i=\dim \mathbf{S}^{1}(\mathcal{N}_{i})_{0}$. \\ \indent
La suite exacte longue associée au triangle distingué
$$
\mathbf{S}(\mathcal{N}_i)_{0} \longrightarrow  \mathbf{S}(\mathcal{M}_i(\ast D) )_{0}\longrightarrow    \mathbf{S}(\mathcal{M}_i)_{0}\overset{+1}{\longrightarrow} 
$$  
s'écrit donc
\begin{equation}\label{suitexactelongue1}
0\longrightarrow \mathbf{S}^{0}(\mathcal{M}_i)_{0}
\longrightarrow \mathbf{S}^{1}(\mathcal{N}_i)_{0}
\longrightarrow \mathbf{S}^{1}(\mathcal{M}_i(\ast D) )_{0}
\longrightarrow  \mathbf{S}^{1}(\mathcal{M}_i)_{0}
\longrightarrow 0
\end{equation}
et donne les relations
\begin{equation}\label{unerelation}
m_i-n_i+\irr_i-\mu_i=0
\end{equation}
D'après \eqref{irrdiagetponct}, on a par hypothèse
\begin{equation}\label{ptiiso}
\Irr^{\ast}_{(0,0)}(\mathcal{M}_{1}\boxtimes \mathcal{M}_{2})\simeq \Irr^{\ast}_{0}(\mathcal{M}_{1}\otimes^{\mathds{L}}_{\mathcal{O}_{ \mathds{C}}}\mathcal{M}_{2}) 
\end{equation}
Or l'irrégularité en $0$ de  $\mathcal{M}_{1}\otimes^{\mathds{L}}_{\mathcal{O}_{ \mathds{C}}}\mathcal{M}_{2}$ est aussi l'irrégularité en $0$ du complexe $\mathcal{M}_{1}(\ast 0)\otimes^{\mathds{L}}_{\mathcal{O}_{ \mathds{C}}}\mathcal{M}_{2}(\ast 0) $. Puisque chacun des $\mathcal{M}_{i}(\ast 0)$ est plat sur $\mathcal{O}_{ \mathds{C}}$, on a 
\begin{equation}\label{concdegre0}
\mathcal{M}_{1}(\ast 0)\otimes^{\mathds{L}}_{\mathcal{O}_{ \mathds{C}}}\mathcal{M}_{2}(\ast 0) \simeq \mathcal{M}_{1}(\ast 0)\otimes_{\mathcal{O}_{ \mathds{C}}}\mathcal{M}_{2}(\ast 0) 
\end{equation}  
et alors le membre de gauche de \eqref{concdegre0}
est concentré en degré $0$. On en déduit que $\Irr^{\ast}_{0}(\mathcal{M}_{1}\otimes^{\mathds{L}}_{\mathcal{O}_{ \mathds{C}}}\mathcal{M}_{2}) $ et par suite 
 $\Irr^{\ast}_{(0,0)}(\mathcal{M}_{1}\boxtimes \mathcal{M}_{2})$ est concentré en degré $1$. \\ \indent
 En particulier le triangle \eqref{triangle2} donne 
\begin{equation}\label{Hdeuxnul}
\mathcal{H}^{2}\Irr^{\ast}_{(0,0)}(\mathcal{M}_{1}\boxtimes \mathcal{M}_{2})\simeq 
 \mathbf{S}^{1}(\mathcal{M}_{1})_{0}\otimes \mathbf{S}^{1}(\mathcal{M}_{2})_{0}\simeq 0
\end{equation}
Dans la suite, on peut donc supposer que $\mathbf{S}^{1}(\mathcal{M}_{1})_{0}\simeq 0$, à savoir $\mu_1=0$.\\ \indent
Le triangle \eqref{triangle2} donne aussi 
une suite exacte 
$$
\begin{tikzcd}
  0 \rar & \mathbf{S}^{0}(\mathcal{M}_1)_{0}\otimes  \mathbf{S}^{0}(\mathcal{M}_2)_{0} \rar
             \ar[draw=none]{d}[name=X, anchor=center]{}
    & \mathbf{S}^{0}(\mathcal{N}_1)_{0}\otimes  \mathbf{S}^{0}(\mathcal{N}_2)_{0}   \ar[rounded corners,
            to path={ -- ([xshift=2ex]\tikztostart.east)
                      |- (X.center) \tikztonodes
                      -| ([xshift=-2ex]\tikztotarget.west)
                      -- (\tikztotarget)}]{dll}[at end]{} \\
  \mathcal{H}^{1}\Irr^{\ast}_{(0,0)}(\mathcal{M}_{1}\boxtimes \mathcal{M}_{2})  \rar &\mathbf{S}^{0}(\mathcal{M}_1)_{0}\otimes  \mathbf{S}^{1}(\mathcal{M}_2)_{0}  \rar & 0
\end{tikzcd}
$$
d'où la relation 
\begin{align*}\label{calculdim}
\dim \mathcal{H}^{1}\Irr^{\ast}_{(0,0)}(\mathcal{M}_{1}\boxtimes \mathcal{M}_{2})&=m_1\mu_2+n_1 n_2-m_1 m_2 \\ 
 &=m_1(\mu_2-m_2) +n_1n_2      \\
&=n_2 \irr_1 +m_1 \irr_2 \\
&=n_2 \irr_1 +n_1 \irr_2-\irr_1 \irr_2
\end{align*}
où les deux dernières égalités proviennent de \eqref{unerelation}. 
A l'aide de \eqref{ptiiso}, on obtient
$$
n_2 \irr_1 +n_1 \irr_2=\irr_1 \irr_2+ \irr_{0}(\mathcal{M}_{1}(\ast 0)\otimes \mathcal{M}_{2}(\ast 0))
$$
D'après \eqref{majorationn}, on a d'une part
$$
n_2 \irr_1 +n_1 \irr_2\leq \rg \hat{\mathcal{M}_2}(\ast 0)^{\reg}\irr_1 +\rg \hat{\mathcal{M}_1}(\ast 0)^{\reg}\irr_2
$$
Comme l'irrégularité est un invariant  de nature formelle, on a d'autre part
\begin{align*}
&\irr_{0}(\mathcal{M}_{1}(\ast 0)\otimes \mathcal{M}_{2}(\ast 0))
\\ &=\irr_{0}(\hat{\mathcal{M}_{1}}(\ast 0)\otimes  \hat{\mathcal{M}_{2}}(\ast 0))\\
                     &=\irr_{0}((\hat{\mathcal{M}_{1}}(\ast 0)^{\reg}\oplus \hat{\mathcal{M}_{1}}(\ast 0)^{\irr})\otimes(\hat{\mathcal{M}_{2}}(\ast 0)^{\reg}\oplus \hat{\mathcal{M}_{2}}(\ast 0)^{\irr}))\\
                     &\geq \irr_{0}(\hat{\mathcal{M}_{1}}(\ast 0)^{\reg}\otimes  \hat{\mathcal{M}_{2}}(\ast 0)^{\irr})+\irr_{0}(\hat{\mathcal{M}_{1}}(\ast 0)^{\irr}\otimes \hat{\mathcal{M}_{2}}(\ast 0)^{\reg}) \\
                     &=\rg \hat{\mathcal{M}_2}(\ast 0)^{\reg}\irr_1 +\rg \hat{\mathcal{M}_1}(\ast 0)^{\reg}\irr_2
\end{align*}
On en déduit que $\irr_1$ ou $\irr_2$ doit être nul, ce qui conclut la preuve de \ref{uneprop}.
\end{proof}

\subsection{Régularité en dehors d'un ensemble discret}\label{regensdisct}
On sait qu'en dehors d'un ensemble discret de points de $X$ passe une hypersurface lisse non caractéristique pour $\mathcal{M}$ et $\mathcal{M}\otimes^{\mathds{L}}_{\mathcal{O}_{X}}\mathcal{M}^{\vee}$. Soit $Z$ une telle hypersurface. D'après \ref{restricara} et la commutation de la dualité avec l'image inverse non caractéristique \cite[2.7.1]{HTT}, on a 
$$\Irr^{\ast}_{\Delta_{Z}}((i_Z^{+}\mathcal{M})^{\vee}\boxtimes i_Z^{+}\mathcal{M})\simeq 0$$
Par hypothèse de récurrence, il vient que  $i_Z^{+}\mathcal{M}$ est régulier. 
D'après \ref{restrcar}, on en déduit pour tout point $x\in Z$
$$
\Irr^{\ast}_{x}(\mathcal{M})\simeq \Irr^{\ast}_{x}(i_Z^{+}\mathcal{M})\simeq 0
$$
Puisque la régularité se teste ponctuellement \cite[6.2-6]{Mehbsmf}, on en déduit que $\mathcal{M}$ est régulier sur le complémentaire d'un ensemble discret de points de $X$.
\\ \indent
Dans la suite, on peut ainsi supposer que $\mathcal{M}$ (et donc $\mathcal{M}^{\vee}$) est régulier en dehors d'un point $x_0$ de $X$. 
\subsection{Réduction au cas à support une courbe}\label{reduction}
Le but de ce paragraphe est de montrer la
\begin{proposition}\label{reduprop}
Il existe une courbe $C$ passant par $x_0$ avec $R\mathcal{M}(\ast C)$ régulier\footnote{En particulier, $R\mathcal{M}^{\vee}(\ast C)$ est régulier d'après \ref{pticoro}.} et 
$$
\Irr^{\ast}_{\Delta_{X}}(R\Gamma_{[C]}\mathcal{M}^{\vee}\boxtimes R\Gamma_{[C]}\mathcal{M}) \simeq 0
$$ 
\end{proposition}
\begin{proof}
Supposons que sur un voisinage de $x_0$, le support de 
$\mathcal{M}$ est contenu dans une courbe $C$. Alors il en est de même pour $\Supp \mathcal{M}^{\vee}$, et donc d'après \ref{coholocalsupport} on a $\mathcal{M}(\ast C)\simeq \mathcal{M}^{\vee}(\ast C)\simeq 0$ et en particulier $R\Gamma_{[C]}\mathcal{M}\simeq \mathcal{M}$ et $R\Gamma_{[C]}\mathcal{M}^{\vee}\simeq \mathcal{M}^{\vee} $ . Les conclusions de \ref{reduprop} sont donc automatiquement vérifiées.\\ \indent
Supposons que par $x_0$ passe une composante de $\Supp \mathcal{M}$ de dimension au moins 2. Posons $Z_0=X$. On se donne une hypersurface $Z_1$ de  $X$ comme en \ref{lemmemeb} pour $\mathcal{M}^{\vee}\oplus \mathcal{M}$ et passant par $x_0$. D'après \ref{regensdisct}, l'irrégularité $\Irr^{\ast}_{Z_1}(\mathcal{M}^{\vee}\oplus \mathcal{M})$ est nulle en dehors de $x_0$. Le critère fondamental de la régularité \ref{corrcriterfond} stipule que $(\mathcal{M}^{\vee}\oplus \mathcal{M})(\ast Z_1)$ est régulier.  \\ \indent
Supposons donc construites des hypersurfaces $Z_1, \dots, Z_k$ de $X$ passant par $x_0$ et telles que pour tout $1\leq i\leq k$, dans les triangles de cohomologie locale\footnote{On note ici et dans la suite $Z_{1i}$ pour $Z_1\cap\dots\cap Z_{i}$.}
\begin{equation}\label{T11}
R\Gamma_{[Z_{1i}]} (\mathcal{M}^{\vee}\oplus \mathcal{M})\longrightarrow R\Gamma_{[Z_{1i-1}]} (\mathcal{M}^{\vee}\oplus \mathcal{M}) \longrightarrow R\Gamma_{[Z_{1i-1}]} (\mathcal{M}^{\vee}\oplus \mathcal{M})( \ast  Z_{i})
\end{equation}
le terme de gauche a un support de dimension strictement plus petite que la dimension du support du terme central, et le terme de droite est régulier. Une telle hypersurface $Z_i$ pour $i\geq 1$ peut toujours être construite suivant le raisonnement précédent tant que $\Supp R\Gamma_{[Z_{1i-1}]}(\mathcal{M}^{\vee}\oplus \mathcal{M})$ a une composante de dimension au moins $2$ passant par $x_0$. On peut donc supposer que 
l'entier $k$ est tel que la dimension des composantes de
$\Supp R\Gamma_{[Z_{1k}]}(\mathcal{M}^{\vee}\oplus \mathcal{M})$ passant par $x_0$ est $\leq 1$. \\ \indent
Par construction des $Z_i$, les triangles
\begin{equation}\label{T1}
R\Gamma_{[Z_{1i}]} \mathcal{M}\longrightarrow R\Gamma_{[Z_{1i-1}]} \mathcal{M} \longrightarrow R\Gamma_{[Z_{1i-1}]} \mathcal{M}( \ast  Z_{i})
\end{equation}
et 
\begin{equation}\label{T2}
R\Gamma_{[Z_{1i}]} \mathcal{M}^{\vee}\longrightarrow R\Gamma_{[Z_{1i-1}]} \mathcal{M}^{\vee} \longrightarrow R\Gamma_{[Z_{1i-1}]} \mathcal{M}^{\vee}( \ast  Z_{i})
\end{equation}
satisfont aux hypothèses du lemme \ref{lemme utile}. On en déduit de proche en proche
$$
\Irr^{\ast}_{\Delta_{X}}(R\Gamma_{[Z_{1k}]}\mathcal{M}^{\vee}\boxtimes R\Gamma_{[Z_{1k}]}\mathcal{M}) \simeq 0
$$
Considérons une courbe  $C\subset Z_{1k}$ contenant le support de  $R\Gamma_{[Z_{1k}]}(\mathcal{M}^{\vee}\oplus \mathcal{M})$. D'après \ref{coholocalsupport}, on a 
$$
R\Gamma_{[Z_{1k}]}\mathcal{M} \simeq R\Gamma_{[C]}\mathcal{M} \text{ \quad et \quad }
R\Gamma_{[Z_{1k}]}\mathcal{M}^{\vee} \simeq R\Gamma_{[C]}\mathcal{M}^{\vee}
$$
Il vient ainsi
$$
\Irr^{\ast}_{\Delta_{X}}(R\Gamma_{[C]}\mathcal{M}^{\vee}\boxtimes R\Gamma_{[C]}\mathcal{M}) \simeq 0
$$
Pour montrer \ref{reduprop}, il reste à montrer que $R\mathcal{M}( \ast  C)$ est régulier. 
Pour $i=1,\dots, k$ l'inclusion $C\subset Z_i$ donne que le triangle de cohomologie locale pour le complexe $R\Gamma_{[Z_{1i-1}]} R \mathcal{M}(\ast  C)$ et l'hypersurface $Z_i$ s'écrit
\begin{equation}\label{encoreuntriangle}
\xymatrix{
R\Gamma_{[Z_{1i}]} R \mathcal{M}(\ast  C)\ar[r]&  R\Gamma_{[Z_{1i-1}]} R \mathcal{M}(\ast  C)\ar[r]&  R\Gamma_{[Z_{1i-1}]}  \mathcal{M}( \ast  Z_{i})\ar[r]^-{+1}& 
}
\end{equation}
En particulier, le troisième terme de \eqref{encoreuntriangle} est régulier. On voit de proche en proche que pour montrer que $R\mathcal{M}( \ast  C)$ est régulier, il suffit de montrer que $R\Gamma_{[Z_{1k}]} R\mathcal{M}(\ast  C)\simeq R(R\Gamma_{[Z_{1k}]}  \mathcal{M})(\ast  C)\simeq 0$ est régulier, ce qui est bien le cas.
\end{proof}
Comme conséquence de \ref{reduprop}, on voit en utilisant les triangles \eqref{T1} et \eqref{T2} que le théorème \ref{theo2prime} découle du 
\begin{soustheorem}\label{theo3}
Soit $\mathcal{M}\in D^{b}_{\hol}(\mathcal{D}_X)$ supposé régulier en dehors de $x_0$. Soit $C$ une courbe passant par $x_0$ pour laquelle $R\mathcal{M}(\ast C)$ est régulier. On suppose de plus que 
$$
\Irr^{\ast}_{\Delta_{X}}(R\Gamma_{[C]}\mathcal{M}^{\vee}\boxtimes R\Gamma_{[C]}\mathcal{M}) \simeq 0
$$ 
Alors, le complexe $R\Gamma_{[C]}\mathcal{M}^{\vee}$ ou le complexe $R\Gamma_{[C]}\mathcal{M}$ est régulier.
\end{soustheorem}

\subsection{Elimination des composantes non pertinentes de $C$}\label{elimination}
Dans la suite, on posera $\mathcal{M}_1:=R\Gamma_{[C]}\mathcal{M}^{\vee}$, $\mathcal{M}_2:=R\Gamma_{[C]}\mathcal{M}$ et on 
utilisera librement les notations de \ref{notations} pour $\mathcal{M}_1$ et $\mathcal{M}_2$. On a par hypothèse
\begin{equation}\label{encfselkfdjl}
\Irr^{\ast}_{\Delta_{X}}(\mathcal{M}_1\boxtimes \mathcal{M}_2) \simeq 0
\end{equation} 
On dispose pour tout $i$ et pour $k=1,2$ d'un morphisme d'adjonction
\begin{equation}\label{adjof}
p_+p^{\dag}\mathcal{H}^{i}\mathcal{M}_k \longrightarrow \mathcal{H}^{i}\mathcal{M}_k 
\end{equation}
qui est un isomorphisme en dehors de $x_{0}$. En particulier le noyau et le conoyau de \eqref{adjof} sont à support $x_0$. Notons $\mathcal{M}_{k,i,l}$ la restriction de  $p^{\dag}\mathcal{H}^{i}\mathcal{M}_k$ à la composante $\tilde{C}_{l}$ de $\tilde{C}$ et faisons l'hypothèse que les  $\mathcal{M}_{2,i,n}$ sont réguliers pour tout $i$. On va montrer qu'alors $\mathcal{M}$ vérifie les conditions de \ref{theo3} avec $C$ remplacée par $C^{\prime}:=C_1\cup \dots \cup C_{n-1}$. Cela revient à montrer la 
\begin{proposition}\label{redsup}
Le complexe $R\mathcal{M}( \ast  C^{\prime})$ est régulier, et on a 
\begin{equation}\label{iso1}
\Irr^{\ast}_{\Delta_{X}}(R\Gamma_{[C^{\prime}]}\mathcal{M}^{\vee}\boxtimes R\Gamma_{[C^{\prime}]}\mathcal{M}) \simeq 0
\end{equation}
\end{proposition}
\begin{proof}
Par application du foncteur de localisation le long de $C^{\prime}$ au triangle
$$
\xymatrix{
R\Gamma_{[C]}\mathcal{M}\ar[r] &   \mathcal{M} \ar[r] &   R\mathcal{M}(\ast C)  \ar[r]^-{+1}&
}
$$
on obtient avec les notations qui précèdent un triangle
$$
\xymatrix{
R\mathcal{M}_{2}(\ast C^{\prime})\ar[r] &   R\mathcal{M}(\ast C^{\prime})   \ar[r] &   R\mathcal{M}(\ast C)  \ar[r]^-{+1}&
}
$$
dont le dernier terme est régulier. Pour montrer que $R\mathcal{M}(\ast C^{\prime})$  est régulier, il suffit donc de montrer que $R\mathcal{M}_{2}(\ast C^{\prime})$ est régulier. On déduit de la suite spectrale
$$
\mathcal{H}^{j}R(\mathcal{H}^{i}\mathcal{M}_{2})(\ast C^{\prime}) \Longrightarrow \mathcal{H}^{i+j}R\mathcal{M}_{2}(\ast C^{\prime})
$$
qu'il suffit pour cela de montrer que les $R(\mathcal{H}^{i}\mathcal{M}_{2})(\ast C^{\prime})$ sont réguliers. \\ \indent
Or en appliquant le foncteur de localisation le long de $C^{\prime}$ à \eqref{adjof}, on obtient un isomorphisme
\begin{equation}\label{unptiiso}
R(p_+p^{\dag}\mathcal{H}^{i}\mathcal{M}_{2})(\ast C^{\prime})\overset{\sim}{\longrightarrow} R(\mathcal{H}^{i}\mathcal{M}_{2})(\ast C^{\prime})
\end{equation}
Il suffit donc de montrer que le terme de gauche de \eqref{unptiiso} est régulier. Par propreté de $p$, on a d'après \cite[3.6-4]{Mehbsmf}  
$$
R(p_+p^{\dag}\mathcal{H}^{i}\mathcal{M}_{2})(\ast C^{\prime})\simeq p_{+}\left[R(p^{\dag}\mathcal{H}^{i}\mathcal{M}_2)(\ast p^{-1}(C^{\prime}))\right]
$$
donc par préservation de la régularité par image directe, il suffit de montrer que $R(p^{\dag}\mathcal{H}^{i}\mathcal{M}_2)(\ast p^{-1}(C^{\prime}))$ est régulier.  Or on a $p^{-1}(C^{\prime})=\tilde{C}_1\cup \dots \cup \tilde{C}_{n-1} \cup \{x_{n}\}$, donc le complexe $R(p^{\dag}\mathcal{H}^{i}\mathcal{M}_2)(\ast p^{-1}(C^{\prime}))$ est à support $\tilde{C}_n$. En particulier, il est régulier si et seulement si sa restriction à $\tilde{C}_n$ l'est. Or cette restriction est justement $\mathcal{M}_{2,i,n}(\ast x_n)$ qui est régulier par hypothèse. La régularité de $R\mathcal{M}(\ast C^{\prime})$  est donc prouvée. D'après \ref{pticoro}, le complexe $R\mathcal{M}^{\vee}(\ast C^{\prime}) $ est régulier. En particulier, $R\mathcal{M}_{1}(\ast C^{\prime})$ et $R\mathcal{M}_{2}(\ast C^{\prime})$ sont réguliers. 
\\ \indent
Dans les triangles 
$$
\xymatrix{
R\Gamma_{[C^{\prime}]}\mathcal{M}^{\vee}\ar[r] &  \mathcal{M}_{1}  \ar[r] &   R\mathcal{M}_{1}(\ast C^{\prime})  \ar[r]^-{+1}&
}
$$
et
$$
\xymatrix{
R\Gamma_{[C^{\prime}]}\mathcal{M}\ar[r] &  \mathcal{M}_{2}  \ar[r] &   R\mathcal{M}_{2}(\ast C^{\prime})  \ar[r]^-{+1}&
}
$$
on sait donc que  $R\mathcal{M}_1(\ast C^{\prime})$ et $R\mathcal{M}_2(\ast C^{\prime})$ sont réguliers. Or
$$
R\Gamma_{[C^{\prime}]}\mathcal{M}\otimes^{\mathds{L}}_{\mathcal{O}_{X}}R\mathcal{M}_{1}(\ast C^{\prime})  \simeq 0$$
et
$$
R\Gamma_{[C^{\prime}]}\mathcal{M}^{\vee}\otimes^{\mathds{L}}_{\mathcal{O}_{X}} R\mathcal{M}_{2}(\ast C^{\prime}) \simeq 0$$
L'annulation \eqref{iso1} est donc une conséquence de \eqref{encfselkfdjl} et \ref{lemme utile}.
\end{proof}
Si $\mathcal{M}_{2,i,l}$ est régulier pour tout entier $i$ et $l$, alors par préservation de la régularité par image directe on obtient que $p_+p^{\dag}\mathcal{H}^{i}\mathcal{M}_2$ est régulier pour tout $i$. Puisque le noyau et le conoyau de \eqref{adjof}  sont à support $x_0$, il vient que $\mathcal{H}^{i}\mathcal{M}_2$ est régulier pour tout $i$ et alors $\mathcal{M}_2$ est régulier et le théorème \ref{theo3} est acquis. \\ \indent
Sinon d'après la proposition \ref{redsup}, on peut supposer que pour \textbf{toute} composante $C_l$ de $C$, il existe un entier $i_l$ avec $\mathcal{M}_{2,i_l,l}$ non régulier.\\ \indent
De la même façon, si $\mathcal{M}_{1,i,l}$ est régulier pour tout entier $i$ et $l$, le théorème \ref{theo3} est acquis. Dans le cas contraire, et quitte à retirer de nouveau certaines composantes à $C$, on peut supposer que pour tout $l$, il existe un entier $j_l$ avec $\mathcal{M}_{1,j_l ,l}$ non régulier. 
\subsection{Réduction du théorème \ref{theo3} au cas de la dimension 1.}\label{reddim1}
On commence par le 
\begin{lemme}\label{endroelsem}
Pour tout couple d'entiers $(i,j)$, on a l'annulation
\begin{equation}\label{ann}
\Irr^{\ast}_{\Delta_{X}}(\mathcal{H}^{i}\mathcal{M}_1
\boxtimes \mathcal{H}^{j}\mathcal{M}_2) \simeq 0
\end{equation}
\end{lemme}
\begin{proof}
Du fait de l'identification
$$
\mathcal{H}^{q}(\mathcal{M}_1
\boxtimes \mathcal{M}_2)\simeq \bigoplus_{i+j=q}\mathcal{H}^{i}\mathcal{M}_1\boxtimes \mathcal{H}^{j}\mathcal{M}_2
$$
La proposition \ref{lemme0} implique que les termes  $E_2^{pq}$ éventuellement non nuls de la suite spectrale 
\begin{equation}\label{ss}
E_{2}^{pq}=\mathcal{H}^{p}\Irr^{\ast}_{\Delta_{X}}(\mathcal{H}^{-q}(\mathcal{M}_1
\boxtimes \mathcal{M}_2))  \Longrightarrow \mathcal{H}^{p+q}\Irr^{\ast}_{\Delta_{X}}(\mathcal{M}_1
\boxtimes \mathcal{M}_2)
\end{equation}
sont situés sur les droites $p=2d-1$ et $p=2d$. Ainsi, \eqref{ss} dégénère en page $2$.
En particulier pour $i$ et $j$ fixés, les espaces $$\mathcal{H}^{2d-1}\Irr^{\ast}_{\Delta_{X}}(\mathcal{H}^{i}\mathcal{M}_1
\boxtimes \mathcal{H}^{j}\mathcal{M}_2)  \text{\quad et \quad} 
 \mathcal{H}^{2d}\Irr^{\ast}_{\Delta_{X}}(\mathcal{H}^{i}\mathcal{M}_1
\boxtimes \mathcal{H}^{j}\mathcal{M}_2) $$ sont des sous-quotients des $\mathcal{H}^{p+q}\Irr^{\ast}_{\Delta_{X}}(\mathcal{M}_1
\boxtimes \mathcal{M}_2)$. On en déduit que \eqref{encfselkfdjl} entraine les annulations \eqref{ann}.
\end{proof}
En particulier pour tout entier $l$, on a 
$$
\Irr^{\ast}_{\Delta_{X}}(\mathcal{H}^{i_{l}}\mathcal{M}_1
\boxtimes \mathcal{H}^{j_{l}}\mathcal{M}_2) \simeq 0
$$
avec $\mathcal{M}_{1,i_l, l}$ et $\mathcal{M}_{2,j_l, l}$ non réguliers. \\ \indent
On se donne une hypersurface $D$ de $X$ passant par $x_0$ et ne contenant aucune composante irréductible de $C$. Par application de $p^{\dag}$ au triangle de cohomologie locale
$$
\xymatrix{
R\Gamma_{[D]}\mathcal{H}^{i_l}\mathcal{M}_1\ar[r] &   \mathcal{H}^{i_l}\mathcal{M}_1\ar[r] &   \mathcal{H}^{i_l}\mathcal{M}_1(\ast D)  \ar[r]^-{+1}&
}
$$
on obtient un triangle dont la restriction à la composante $\tilde{C}_l$ de $\tilde{C}$ s'écrit
\begin{equation}\label{qsmldfk}
\xymatrix{
R\Gamma_{[x_l]}\mathcal{M}_{1,i_l, l}\ar[r] &  \mathcal{M}_{1,i_l, l}\ar[r] &   \mathcal{M}_{1,i_l, l}(\ast x_l)  \ar[r]^-{+1}&
}
\end{equation}
En particulier, on a
\begin{align*}\label{calculker}
\mathcal{H}^{0}R\Gamma_{[x_l]}\mathcal{M}_{1,i_l, l}
&\simeq (\mathcal{H}^{0}p^{\dag}R\Gamma_{[D]}\mathcal{H}^{i_l}\mathcal{M}_1)_{|\tilde{C}_l} \\
& \simeq (\mathcal{H}^{0}p^{\dag}\mathcal{H}^{0}R\Gamma_{[D]}\mathcal{H}^{i_l}\mathcal{M}_1)_{|\tilde{C}_l}\\ 
 &\simeq (p^{\dag}\mathcal{H}^{0}R\Gamma_{[D]}\mathcal{H}^{i_l}\mathcal{M}_1)_{|\tilde{C}_l}
\end{align*}
où la seconde identification provient de ce que les termes éventuellement non nuls de la suite spectrale 
$$
E_2^{pq}=\mathcal{H}^{p}p^{\dag}\mathcal{H}^{-q}R\Gamma_{[D]}\mathcal{H}^{i_l}\mathcal{M}_1 \Longrightarrow \mathcal{H}^{p+q}p^{\dag}R\Gamma_{[D]}\mathcal{H}^{i_l}\mathcal{M}_1
$$
sont d'après \ref{concdegré0} les termes $E_2^{0,0}$ et  $E_2^{0,-1}$, et où la dernière identification provient aussi de \ref{concdegré0}. On en déduit
\begin{align*}
(p^{\dag}
(\mathcal{H}^{i_l}\mathcal{M}_1/\mathcal{H}^{0}R\Gamma_{[D]}\mathcal{H}^{i_l}\mathcal{M}_1
))_{|\tilde{C}_l}&\simeq   
(p^{\dag}\mathcal{H}^{i_l}\mathcal{M}_1/p^{\dag}\mathcal{H}^{0}R\Gamma_{[D]}\mathcal{H}^{i_l}\mathcal{M}_1
)_{|\tilde{C}_l}\\
&\simeq \mathcal{M}_{1,i_l ,
l}/\mathcal{H}^{0}R\Gamma_{[x_l]}\mathcal{M}_{1,i_l, l}
\end{align*}
En particulier, le module $(p^{\dag}
(\mathcal{H}^{i_l}\mathcal{M}_1/\mathcal{H}^{0}R\Gamma_{[D]}\mathcal{H}^{i_l}\mathcal{M}_1
))_{|\tilde{C}_l}$ est nécessairement irrégulier, car dans le cas contraire le triangle \eqref{qsmldfk} montre que 
$\mathcal{M}_{1,i_l, l}(\ast x_l)  $ est régulier, d'où on déduit que $\mathcal{M}_{1,i_l, l}$ est régulier, ce qui est exclu par hypothèse. En faisant de même avec $\mathcal{H}^{j_{l}}\mathcal{M}_2$ et en utilisant \ref{petite réduction}, on a ainsi obtenu en prenant $l=1$ deux modules holonomes de nouveau notés $\mathcal{M}_1$ et $\mathcal{M}_2$ par abus et vérifiant 
\begin{enumerate}
\item pour $k=1,2$ la flèche $\mathcal{M}_k \longrightarrow  \mathcal{M}_k(\ast D)$ est injective.
\item pour $k=1,2$ le module $(p^{\dag}\mathcal{M}_k)_{|\tilde{C}_1}$ est irrégulier. 
\item $\Irr^{\ast}_{\Delta_{X}}(\mathcal{M}_1
\boxtimes \mathcal{M}_2) \simeq 0$.
\end{enumerate}
On va montrer que la condition $(2)$ est contradictoire avec $(1)$ et $(3)$. On a pour $k=1,2$ une suite exacte
\begin{equation}\label{suiteexaccour}
0\longrightarrow \mathcal{M}_k \longrightarrow  \mathcal{M}_k(\ast D)
\longrightarrow   \mathcal{N}_{k}\longrightarrow 0
\end{equation}
avec $\mathcal{N}_{k}$ à support $x_0$. En particulier, il s'agit d'une somme d'un nombre fini de copies du Dirac $\delta$ en $x_0$. On sait donc d'après \ref{concdegré0} que $p^{\dag}\mathcal{N}_{k}$ est concentré en degré $0$. Puisque c'est aussi le cas de $p^{\dag}(\mathcal{M}_k(\ast D))$, il vient que $p^{\dag}\mathcal{M}_k$ est concentré en degré $0$. D'autre part, le morphisme $p$ étant fini, le foncteur $p_{+}$ est exact. Par application de $p_{+}p^{\dag}$ à la suite exacte courte \eqref{suiteexaccour}, on en déduit que 
$p_+p^{\dag}\mathcal{M}_k$ est concentré en degré $0$ et que 
le morphisme supérieur du carré 
\begin{equation}\label{carré}
\xymatrix{
    p_+p^{\dag} \mathcal{M}_k \ar[r] \ar[d] & p_+p^{\dag} (\mathcal{M}_k(\ast D)) \ar[d]^-\wr \\
    \mathcal{M}_k  \ar@{^{(}->}[r]& \mathcal{M}_k(\ast D) 
}
\end{equation}
est une injection, l'isomorphisme de droite provenant de \eqref{identification}. En particulier, le morphisme d'adjonction 
  $p_+p^{\dag}\mathcal{M}_k \longrightarrow \mathcal{M}_k $ est injectif. Comme par théorème de Kashiwara, il s'agit d'un isomorphisme en dehors de $x_0$, son conoyau noté $\mathcal{K}_{k}$ est une somme d'un nombre fini de copies de $\delta$. \\ \indent
Par application de $\Irr^{\ast}_{\Delta_{X}}$ au diagramme 
\begin{equation}\label{gros diagramme2}
\xymatrix{
p_+p^{\dag}\mathcal{M}_1  \boxtimes  
p_+p^{\dag}\mathcal{M}_2  \ar[r] \ar[d]&              \mathcal{M}_1  \boxtimes  
p_+p^{\dag}\mathcal{M}_2 \ar[r] \ar[d]&              \mathcal{K}_{1}  \boxtimes  
p_+p^{\dag}\mathcal{M}_2\ar[d] \\
p_+p^{\dag}\mathcal{M}_1  \boxtimes
\mathcal{M}_2 
\ar[r] \ar[d]&    
\mathcal{M}_1  \boxtimes   \mathcal{M}_2     \ar[r] \ar[d]&            \ar[d] \mathcal{K}_{1}  \boxtimes  \mathcal{M}_2\\
p_+p^{\dag}\mathcal{M}_1  \boxtimes \mathcal{K}_{2} \ar[r] &  
\mathcal{M}_1  \boxtimes  \mathcal{K}_{2}
\ar[r] &            \mathcal{K}_{1}  \boxtimes    \mathcal{K}_{2}
}
\end{equation}
on obtient un diagramme 
$$
\xymatrix@C=0.4cm{
\Irr^{\ast}_{\Delta_{X}}(p_+p^{\dag}\mathcal{M}_1  \boxtimes  
p_+p^{\dag}\mathcal{M}_2 )  &    \Irr^{\ast}_{\Delta_{X}}(\mathcal{M}_1  \boxtimes  
p_+p^{\dag}\mathcal{M}_2 )      \ar[l] &              
\Irr^{\ast}_{\Delta_{X}}(\mathcal{K}_{1}  \boxtimes  
p_+p^{\dag}\mathcal{M}_2) \ar[l] \\
\Irr^{\ast}_{\Delta_{X}}(p_+p^{\dag}\mathcal{M}_1  \boxtimes
\mathcal{M}_2 )\ar[u]& 0  \ar[l] \ar[u]&     \Irr^{\ast}_{\Delta_{X}}(\mathcal{K}_{1}  \boxtimes  \mathcal{M}_2)  \ar[u]_-{\wr}\ar[l]  \\
\Irr^{\ast}_{\Delta_{X}}(p_+p^{\dag}\mathcal{M}_1  \boxtimes \mathcal{K}_{2}) \ar[u]  &  \Irr^{\ast}_{\Delta_{X}}( \mathcal{M}_1  \boxtimes  \mathcal{K}_{2}) 
\ar[l]_-{\sim} \ar[u]&  0 \ar[u]\ar[l] 
}
$$
appelé $\Diag$ dont on note $\Diag_{ab}$ le terme placé en ligne $a$ et colonne $b$. \\ \indent
D'après \ref{lemme0}, les complexes intervenant dans  $\Diag_{ab}$  sont concentrés en degrés $2d$ et $2d-1$. L'exactitude de la seconde ligne montre que $\Diag_{21}$ est concentré en degré $2d-1$. On en déduit par exactitude de la première colonne que $\mathcal{H}^{2d}\Diag_{11}\simeq 0$. \\ \indent
Or par compatibilité \ref{compmorphismepropre} de l'irrégularité avec les morphismes propres, on a 
$$
\Diag_{11}\simeq (p\times p)_{\ast}\Irr^{\ast}_{(p\times p)^{-1}(\Delta_{X})}(p^{\dag}\mathcal{M}_1  \boxtimes  
p^{\dag}\mathcal{M}_2 ) [2-2d]
$$
En particulier le complexe
$$
 \Irr^{\ast}_{\Delta_{\tilde{C}_1}}((p^{\dag}\mathcal{M}_1)_{|\tilde{C}_1} \boxtimes  
(p^{\dag}\mathcal{M}_2)_{|\tilde{C}_1}  )_{(x_1,x_1)}[2-2d]
$$ 
est un facteur direct de $(\Diag_{11})_{(x_0,x_0)}$.
Comme par théorème de perversité de Mebkhout \cite[2.1.6]{Mehbgro}, l'irrégularité diagonale $\Irr^{\ast}_{\Delta_{\tilde{C}_1}}((p^{\dag}\mathcal{M}_1)_{|\tilde{C}_1} \boxtimes  
(p^{\dag}\mathcal{M}_2)_{|\tilde{C}_1} )$ est un dirac en $(x_1,x_1)$ concentré en degré $2$, on déduit de l'annulation de $\mathcal{H}^{2d}\Diag_{11}$ que 
\begin{equation}\label{annul}
\Irr^{\ast}_{\Delta_{\tilde{C}_1}}((p^{\dag}\mathcal{M}_1)_{|\tilde{C}_1}  \boxtimes  
(p^{\dag}\mathcal{M}_2)_{|\tilde{C}_1})\simeq 0
\end{equation}
D'après \ref{uneprop}, ceci est contradictoire avec le fait que les modules $(p^{\dag}\mathcal{M}_1)_{|\tilde{C}_1}$ et $(p^{\dag}\mathcal{M}_2)_{|\tilde{C}_1}$ sont irréguliers.

\section{Deux lemmes techniques}
\subsection{} Dans ce paragraphe, on se donne une variété complexe $X$, un point $x\in X$ et deux triangles de $D^{b}_{\hol}(\mathcal{D}_X)$
$$
\xymatrix{
\mathcal{M}_{1}\ar[r] &  \mathcal{M}_{2}\ar[r]&  \mathcal{M}_{3}\ar[r]^-{+1}&
}
$$
et 
$$
\xymatrix{
\mathcal{N}_{1}\ar[r] &  \mathcal{N}_{2}\ar[r]&  \mathcal{N}_{3}\ar[r]^-{+1}&
}
$$
tels que 
\begin{enumerate}
\item $\mathcal{M}_{3}$ et $\mathcal{N}_{3}$ sont réguliers.
\item Les complexes $\mathcal{M}_{3}\otimes^{\mathds{L}}_{\mathcal{O}_{X}}\mathcal{N}_{1}$ et $\mathcal{M}_{1}\otimes^{\mathds{L}}_{\mathcal{O}_{X}}\mathcal{N}_{3}$ sont réguliers.
\item $\mathcal{M}_{3}$ et $\mathcal{N}_{3}$ sont localisés en $x$, à savoir $\mathcal{M}_{3}\simeq R\mathcal{M}_{3}(\ast x)$ et $\mathcal{N}_{3}\simeq R\mathcal{N}_{3}(\ast x)$.
\end{enumerate}
On va montrer le 
\begin{lemme}\label{lemme utile}
L'annulation $\Irr^{\ast}_{\Delta_{X}}(\mathcal{M}_{2}\boxtimes \mathcal{N}_{2})_{(x,x)}\simeq 0$ équivaut à  l'annulation \newline $\Irr^{\ast}_{\Delta_{X}}(\mathcal{M}_{1}\boxtimes \mathcal{N}_{1})_{(x,x)}\simeq 0$
\end{lemme}
\begin{proof}
En appliquant $\Irr^{\ast}_{\Delta_{X}}$ au diagramme
$$
\xymatrix{
\mathcal{M}_{1}\boxtimes
\mathcal{N}_{1} \ar[r] \ar[d]& \mathcal{M}_{2}\boxtimes
\mathcal{N}_{1} \ar[r] \ar[d]&   \mathcal{M}_{3}\boxtimes
\mathcal{N}_{1}  \ar[d] \\
\mathcal{M}_{1}\boxtimes
\mathcal{N}_{2} \ar[r] \ar[d]&   \mathcal{M}_{2}\boxtimes
\mathcal{N}_{2} \ar[r] \ar[d]&      \mathcal{M}_{3}\boxtimes
\mathcal{N}_{2}  \ar[d] \\
\mathcal{M}_{1}\boxtimes
\mathcal{N}_{3} \ar[r] &    \mathcal{M}_{2}\boxtimes
\mathcal{N}_{3}
\ar[r] &             \mathcal{M}_{3}\boxtimes
\mathcal{N}_{3}
}
$$
on obtient par hypothèse le diagramme à lignes et colonnes distinguées
$$
\xymatrix{
\Irr^{\ast}_{\Delta_{X}}(\mathcal{M}_{1}\boxtimes
\mathcal{N}_{1})_{(x,x)}  & \Irr^{\ast}_{\Delta_{X}}(\mathcal{M}_{2}\boxtimes
\mathcal{N}_{1})_{(x,x)} \ar[l] &   \Irr^{\ast}_{\Delta_{X}}(\mathcal{M}_{3}\boxtimes
\mathcal{N}_{1} )_{(x,x)} \ar[l] \\
\Irr^{\ast}_{\Delta_{X}}(\mathcal{M}_{1}\boxtimes
\mathcal{N}_{2} )_{(x,x)} \ar[u]&   \Irr^{\ast}_{\Delta_{X}}(\mathcal{M}_{2}\boxtimes \mathcal{N}_{2})_{(x,x)}\ar[l] \ar[u]&      \Irr^{\ast}_{\Delta_{X}}(\mathcal{M}_{3}\boxtimes
\mathcal{N}_{2} )_{(x,x_)} \ar[l]\ar[u] \\
\Irr^{\ast}_{\Delta_{X}}(\mathcal{M}_{1}\boxtimes
\mathcal{N}_{3} )_{(x,x)}\ar[u]&   \Irr^{\ast}_{\Delta_{X}}( \mathcal{M}_{2}\boxtimes
\mathcal{N}_{3})_{(x,x)}
\ar[l] \ar[u]&        0\ar[u] \ar[l]
}
$$
Pour démontrer \ref{lemme utile}, il suffit de démontrer $$\Irr^{\ast}_{\Delta_{X}}(\mathcal{M}_{3}\boxtimes
\mathcal{N}_{1} )_{(x,x)}\simeq 0 \text{\quad et \quad}
\Irr^{\ast}_{\Delta_{X}}(\mathcal{M}_{1}\boxtimes
\mathcal{N}_{3} )_{(x,x)}\simeq 0$$
Par symétrie des hypothèses, il suffit de traiter la première annulation. Par hypothèse, on a
$$
\Irr^{\ast}_{x}(\mathcal{M}_{3}\otimes^{\mathds{L}}_{\mathcal{O}_{X}}\mathcal{N}_{1}) \simeq 0
$$
Le triangle \eqref{irrdiagetponct} donne donc une identification 
$$
\Irr^{\ast}_{\Delta_{X}}(\mathcal{M}_{3}\boxtimes \mathcal{N}_{1})_{(x,x)} \simeq   \Irr^{\ast}_{(x,x)}(\mathcal{M}_{3}\boxtimes \mathcal{N}_{1})
$$
Or $\mathcal{M}_{3}$ est localisé en $x$, donc le triangle \eqref{triangle2} donne
$$
\Irr^{\ast}_{(x,x)}(\mathcal{M}_{3}\boxtimes \mathcal{N}_{1}) \overset{\sim}{\longrightarrow}  \mathbf{S}(\mathcal{M}_{3})_{x}\otimes \mathbf{S}(\mathcal{N}_{1})_{x}
$$
Or on a
$$
\mathbf{S}(\mathcal{M}_{3})_{x}\simeq \mathbf{S}(R\mathcal{M}_{3}(\ast x))_{x}=\Irr^{\ast}_{x}(\mathcal{M}_{3})\simeq 0
$$
donc le lemme \ref{lemme utile} est prouvé.
\end{proof}


\subsection{}
On se place dans les conditions géométriques du point \ref{notations}, dont on adopte les notations. On a le 
\begin{lemme}\label{concdegré0}
Le complexe $p^{\dag}\delta$ est concentré en degré $0$.
\end{lemme}
\begin{proof}
Si $i: x_0 \hookrightarrow X$ désigne l'inclusion de $x_0$ dans $X$, le module $\delta$ est par définition  $i_+\mathds{C}$.  Le diagramme
$$
\xymatrix{
   \{ x_1,\dots, x_n\} \ar[r]^-{p^{\prime}} \ar[d]^-{i^{\prime}} & x_0 \ar[d]^-{i}\\
   \tilde{C} \ar[r]^-{p}& X
}
$$
étant cartésien, le théorème de changement de base \cite[1.7.3]{HTT} permet de voir qu'on a une identification canonique
$$
p^{\dag}\delta\simeq i_{+}^{\prime}p^{\prime +}\mathds{C}
$$
En particulier le complexe $p^{\dag}\delta$ est concentré en degré $0$.
\end{proof}


\bibliographystyle{amsalpha} 
\bibliography{regcopie}

\end{document}